\documentclass[12pt]{amsart}
 \newtheorem{Theorem}{Theorem}[section]
 \newtheorem{corollary}[Theorem]{Corollary}
 \newtheorem{lemma}[Theorem]{Lemma}
 \newtheorem{proposition}[Theorem]{Proposition}
 \theoremstyle{definition}
 \newtheorem{definition}[Theorem]{Definition}
 \theoremstyle{remark}
 \newtheorem{remark}[Theorem]{Remark}
 \newtheorem{example}{Example}
 \numberwithin{equation}{section}

\usepackage[cp1250]{inputenc}
\usepackage{graphicx}
\usepackage{url}
\usepackage[all]{xy}
\usepackage{multicol}
\usepackage{amsmath}
\usepackage{color}
\usepackage{amssymb}
\usepackage{geometry}
\usepackage{pinlabel}
\usepackage{hyperref}

\newcommand{\tr}{\triangleright}
\newcommand{\tl}{\triangleleft}

\newcommand{\qb}{\mathcal{BQ}}

\begin{document}

\title{Knot quandle decompositions}

\author{A. Cattabriga}
\address{
Department of Mathematics \\University of Bologna\\Piazza di Porta San Donato 5\\ 40126 Bologna \\Italy}
\email{alessia.cattabriga@unibo.it}
   
\thanks{The first author was  supported by the ``National Group for Algebraic and Geometric Structures, and their Applications" (GNSAGA-INdAM) and University of Bologna, funds for selected research topics. The second author was supported by the Slovenian Research Agency grant N1-0083.}
 
\author{E. Horvat}
\address{Faculty of Education \\University of Ljubljana\\Kardeljeva plo\v s\v cad 16\\ 1000 Ljubljana\\
   Slovenia}\email{eva.horvat@pef.uni-lj.si}



\date{\today}

\begin{abstract}
 We show that the fundamental quandle defines a functor from the oriented tangle category to a suitably defined quandle category. Given a tangle decomposition of a link $L$, the fundamental quandle of $L$ may be obtained from the fundamental quandles of tangles. We apply this result to derive a presentation of the fundamental quandle of periodic links, composite knots and satellite knots. 
\end{abstract}

\maketitle

\begin{section}{Introduction}
A quandle is an algebraic structure whose axioms express invariance under the three Reidemeister moves. The idea of studying knots using quandles arose in the 1980's by Joyce \cite{DJ} and Matveev \cite{MAT}. In this century, quandle-type invariants were extensively developed and successfully used in the theory of classical and virtual links. Most of these invariants are based on the fundamental quandle, which can be defined for any codimension 2 submanifold \cite{FR}.  \\

The fundamental quandle of a link generalizes the link group and, for prime knots,  is a classifying invariant. Nevertheless, it is quite hard to compute and  its topological meaning, as well as its connection to other link invariants, is not completely understood.  For example, its relation with the Alexander polynomial is well known, but no connection with the Jones polynomial has been described (see \cite{FR}).\\

In this paper we present a categorical approach to the fundamental quandle of knots and links. Namely, we look at links as morphisms in the tangle category, and use the fundamental quandle  to construct a functor from the oriented tangle category to a new category we introduce: the bordered quandle category. The objects of this category are free quandles and a morphism between two objects $F_1$ and $F_2$  is given by a quandle $Q$, together with two quandle homomorphisms $F_1\stackrel{i_{1}}{\rightarrow }Q\stackrel{i_2}{\leftarrow }F_2$. \\

Fundamental quandles of tangles were previously used to study, for example, tangle embeddings \cite{AM}, \cite{NI1} and composite knots \cite{CS}. Nevertheless, no categorical approach was introduced. On the contrary, in \cite{BL} a categorical approach is introduced to study quantum invariants of links. However, the category analyzed there is different from ours: the morphisms are $Q$-colored tangles; these are couples $(\Gamma,\rho)$, where $\Gamma $ is a tangle and $\rho$ is quandle morphism from the fundamental quandle of $\Gamma $ to a fixed quandle $Q$.

Our approach, besides giving a theoretical background to some of the previously mentioned results, reveals to be very flexible and effective in computing the fundamental quandle of knots and links. Indeed, using our functor, the fundamental quandle of any link may be decomposed into simpler quandles that correspond to fundamental quandles of tangles constituting the link. As we show in the paper, this decomposition is useful when dealing with  classes of links admitting special tangle decompositions as periodic links, composite knots or satellite knots. Moreover, it becomes a necessary tool when studying large links with many crossings, whose crossing information is too complicated to compute the quandle invariants in one piece. Finally, our approach may reveal a new way of investigating the relationship between the fundamental quandle of a link and other link invariants that admit a description involving the tangle category, such as the Jones polynomial.\\

The paper is organized as follows. In Section \ref{sec2}, we introduce the source and the target category of the fundamental quandle functor. In Subsection \ref{sub21}, we recall the oriented tangle category and introduce some basic notions that we will need. Subsection \ref{sub22} contains the definition of the bordered quandle category. In Section \ref{sec3}, we recall the definition of the fundamental quandle of a codimension 2 submanifold and prove that the fundamental quandle defines a functor from the oriented tangle category to the bordered quandle category. Section \ref{sec4} contains several applications of the fundamental quandle functor. We begin by computing the fundamental quandle of the closure of a tangle. In Subsection \ref{appl1}, we derive a presentation for the fundamental quandle of a periodic link. In Subsection \ref{appl2}, we focus on composite knots, describe a presentation of their fundamental quandle and discuss colorings of composite knots. In Subsection \ref{appl3} we study satellite knots. Using a suitable tangle decomposition of a satellite knot, we derive a presentation of its fundamental quandle in terms of fundamental quandles of the companion and the embellishment knot and give two explicit calculations. We conclude the paper with some ideas for further research, based on our results (see Subsection \ref{concl}). 
\end{section}

\begin{section}{The oriented tangle category and the bordered quandle category}
\label{sec2}

In this Section, we recall the definition of oriented tangle category and define the bordered quandle category.

\begin{subsection}{Oriented tangle category}
\label{sub21}

Let $D$ be the closed unitary disk in $\mathbb C$. The objects of the \textit{oriented tangle category} $\mathcal{T}^+$ are given as $P_{n}^{\phi }$, where 
\begin{itemize}
\item $P_{n}$ is a (possibly empty) finite ordered sequence of points $P_n=(p_1,\dots,p_n)$ on the real interval  $[-1,1]\subset D$, such that $-1<p_1 <p_2 <\ldots <p_n <1$,
\item $\phi \colon P_n \to \{-1,+1\}$ is a function. If $n=0$, then $\phi \colon \emptyset \to \{-1,+1\}$ is just the trivial inclusion. 
\end{itemize} We may think of the values of $\phi $ as a decoration (a plus or minus) for each of the points in $P_n$. We will  call $\phi (p_{i})$ the \textit{sign} of $p_{i}$, and the function $\phi$ could be thought of as an orientation of $\{p_1,\ldots,p_n\}$. If $\phi$ is the constant function equal to $+1$ (resp. $-1$), we set $P^{\phi}_n=P_n^+$ (resp. $P^{\phi}_n=P_n^-$).

Given two objects $P_{n}^{\phi }$ and $P_{m}^{\psi }$, a morphism from $P_{n}^{\phi }$ to $P_{m}^{\psi }$ is an oriented  $(n,m)$-tangle, that is an oriented   $1$-submanifold $\tau$ of $D \times I$, whose  oriented boundary $\partial\tau$ equals
\begin{eqnarray*}
\sum_{i=1}^{m}\psi(p_i)(p_i,1) - \sum_{i=1}^{n}\phi(p_i)(p_i,0)\;,
\end{eqnarray*}
considered up to ambient  isotopy of the cylinder that keeps $D\times \partial I$ fixed. Note that a morphism from $P_{n}^{\phi }$ to $P_{m}^{\psi }$ exists if only if  $\sum_{i=1}^{m}\psi(p_i)=\sum_{i=1}^{n}\phi(p_i)$. A morphism from $P_{n}^{\phi }$ to $P_{m}^{\psi }$ will be called a $(\phi,\psi)$-tangle or sometimes just an $(n,m)$-tangle when we don't need to refer to orientations.

Given a $(\phi,\psi)$-tangle $\tau$, the \textit{negative} (resp. \textit{positive}) oriented boundary of $\tau$ is defined as   $\partial^-\tau=\{\phi(p_1)(p_1,0),\dots,\phi(p_n)(p_{n},0)\}$ (resp. $\partial^+\tau=\{\psi(p_1)(p_1,1),\dots, \psi(p_m)(p_{m},1)\}$).

Note that with the above convention, a component of a tangle with nonempty boundary has  endpoints in different boundaries if and only if  the endpoints have the same sign. Moreover it is oriented upwards (with respect to the height function on the cylinder) if the sign of the endpoints is $+$, and downwards if the sign is $-$. While a component having both ends in the upper disk (resp. lower disk) goes from a $-$ point to a $+$ point (resp. from a $+$ point to a $-$ point).

Geometrically, the composition $\tau_1\tau_2$ of two tangles consists of the tangle obtained by stacking the tangle $\tau_2$ over $\tau_1$, identifying $\partial^+\tau_1$ with $\partial^-\tau_2$, and rescaling the resulting solid cylinder to $D\times I$. 

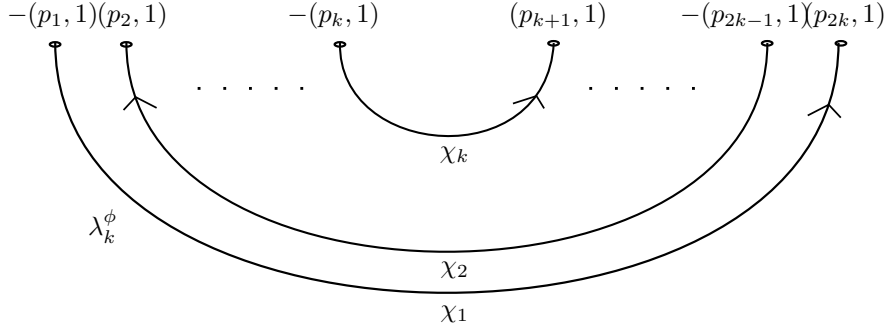
\begin{figure}[h!]
\labellist
\footnotesize \hair 2pt
\pinlabel $\lambda _{k}^{\phi}$ at 98 100
\pinlabel $-(p_1,1)$ at 35 365
\pinlabel $(p_{2},1)$ at 135 365
\pinlabel $-(p_{k},1)$ at 388 365
\pinlabel $(p_{k+1},1)$ at 670 365
\pinlabel $-(p_{2k-1},1)$ at 905 365
\pinlabel $(p_{2k},1)$ at 1030 365
\pinlabel $\chi_1$ at 540 -10
\pinlabel $\chi_2$ at 540 45 
\pinlabel $\chi_k$ at 540 190
\endlabellist
\begin{center}
\includegraphics[scale=0.3]{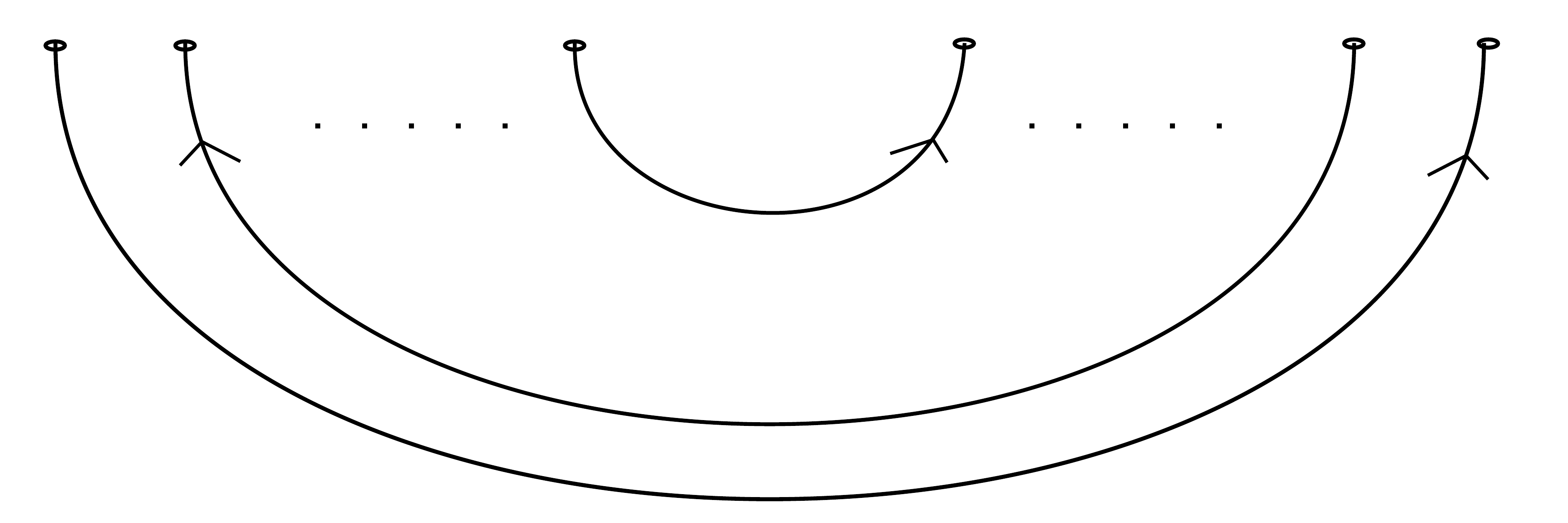}
\caption{The cup tangle $\lambda_k^{\phi}$ of type $\phi$.}
\label{fig_cupcap1}
\end{center}
\end{figure}

Given an object $P_n^{\phi}$, we denote by $P_n^{-\phi}$  the object obtained by changing the signs of all the points. If $\tau$ is a morphism  from $P_{n}^{\phi }$ to $P_{m}^{\psi }$, then $-\tau$ is a morphism from $P_{n}^{-\phi }$ to $P_{m}^{-\psi }$, where $-\tau$ is obtained from $\tau$ by changing the orientation of each component. Moreover, we denote by $\overline{\tau }$ the image of $\tau$ under the homeomorphism $D\times I \to D\times I$, defined by $(z,t)\mapsto (-z,1-t)$, that turns the solid cylinder upside down. Clearly, $\overline{\tau }$  is a morphism from $P_{m}^{\overline{\psi} }$ to $P_{n}^{\overline{\phi} }$, where $\overline{\psi}(p_i)=-\psi(p_{m-i+1})$ for $i=1,\ldots,m$ and  $\overline{\phi}(p_j)=-\phi(p_{n-j+1})$ for $j=1,\ldots,n$. Finally, we denote by $\lambda_k^{\phi}$, $T_k^{\phi}$  and $\mu_k^{\phi}$ the tangles of type $(0,2k)$, $(k,k)$ and $(0,2k)$ respectively, depicted in Figures \ref{fig_cupcap1} and \ref{fig_cupcap2}, and call them the cup tangle, trivial tangle and plat tangle of type ${\phi}$, respectively. Note that for each choice of $\phi$ there is a tangle $T_k^{\phi}$, while the tangle $\lambda_k^{\phi}$ (resp. $\mu_k^{\phi}$) exists if and only if $\phi(p_i)=-\phi(p_{2k-i+1})$ (resp. $\phi(p_{2i})=-\phi(p_{2i-1})$) for $i=1,\ldots,k$. \\

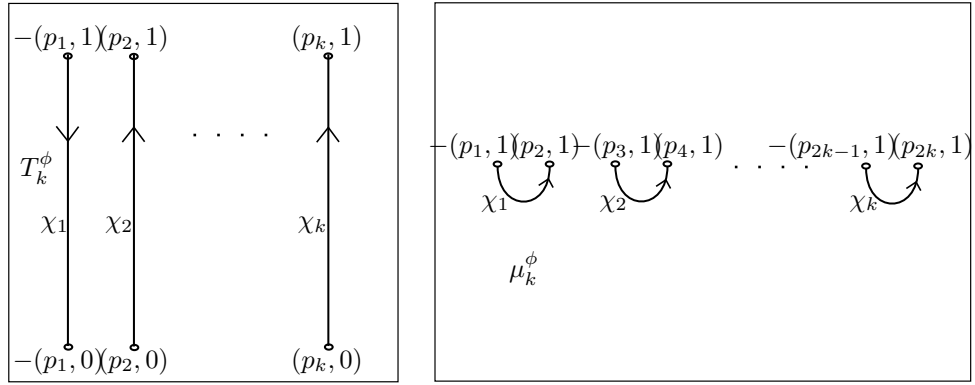
\begin{figure}[h!]
\labellist
\footnotesize \hair 2pt
\pinlabel $T_k^{\phi}$ at 40 300
\pinlabel $\mu _{k}^{\phi}$ at 720 160
\pinlabel $-(p_{1},1)$ at 650 330
\pinlabel $(p_{2},1)$ at 750 330
\pinlabel $-(p_{3},1)$ at 850 330
\pinlabel $(p_{4},1)$ at 950 330
\pinlabel $-(p_{2k-1},1)$ at 1150 330
\pinlabel $(p_{2k},1)$ at 1290 330
\pinlabel $-(p_1,0)$ at 70 30
\pinlabel $(p_{2},0)$ at 175 30
\pinlabel $-(p_{1},1)$ at 70 480
\pinlabel $(p_{k},0)$ at 445 30
\pinlabel $(p_2,1)$ at 175 480
\pinlabel $(p_k,1)$ at 445 480
\pinlabel $\chi_1$ at 65 220
\pinlabel $\chi_2$ at 155 220 
\pinlabel $\chi_k$ at 425 220
\pinlabel $\chi_1$ at 680 250
\pinlabel $\chi_2$ at 845 250  
\pinlabel $\chi_k$ at 1195 250 
\endlabellist
\begin{center}
\includegraphics[scale=0.27]{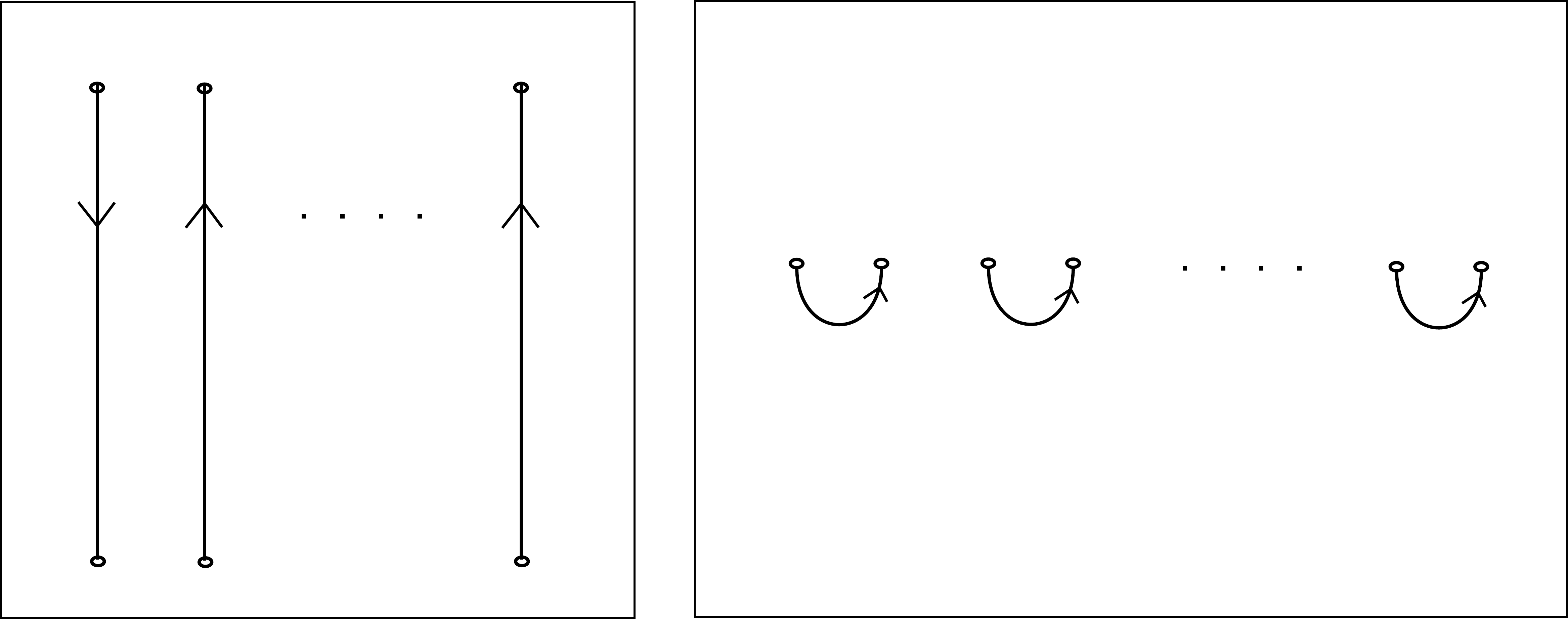}
\caption{The trivial tangle $T_k^{\phi}$ and the plat tangle $\mu_k^{\phi}$ of type $\phi$.}
\label{fig_cupcap2}
\end{center}
\end{figure}

The tangle category is a strict monoidal category, which means it has a tensor-like product. The tensor product of two objects $ P^{\phi}_n\otimes P^{\psi}_m$  (resp. two tangles $\tau_1\otimes\tau_2$)  is the sequence of $n+m$ points (resp. tangle), obtained by juxtaposing $P_m$ (resp. $\tau_2$) to the right of $ P_n$ (resp. $\tau_1$), keeping the signs. This operation is associative and has the empty sequence as the (left and right) identity object.  \\

By forgetting signs and orientation, we obtain a surjective faithful forgetful functor from $\mathcal T^+$ to $\mathcal T$, where $\mathcal T$ is the  \textit{tangle category}, that is  the category whose objects are (possibly empty) finite ordered sequences of points and whose morphisms are tangles. We denote by $P_n$ the image of $P_n^{\phi}$ in $\mathcal T$, while the image of an oriented tangle $\tau$ will still be denoted by $\tau$.

\begin{remark}
\label{rem_closing}
 In the tangle category, a link in $\mathbb R^3$ or $S^3$ can be seen as a  morphism from the empty set to itself, while braids on $k$ strands are isomorphisms of $P_k$. With this in mind, we can rephrase and extend the notions of classical  closure and plat closure of braids to tangles as follows. Given a $(k,k)$-tangle $\tau$, the \textit{classical closure} of $\tau$ is the link $\widehat{\tau}=\lambda_k (\tau\otimes T_k)\overline{\lambda_k}$;  while given a $(2n,2m)$-tangle $\tau $, the \textit{plat closure} of $\tau$ is the link $\widetilde{\tau}=\mu_n\tau\overline{\mu_m}$, see Figure \ref{fig_closingup}. Note that, since $\lambda_1=\mu_1$, for a  $(1,1)$-tangle $\tau$ we have $\widehat{\tau}=\widetilde{\tau\otimes T_1}$. 
 
 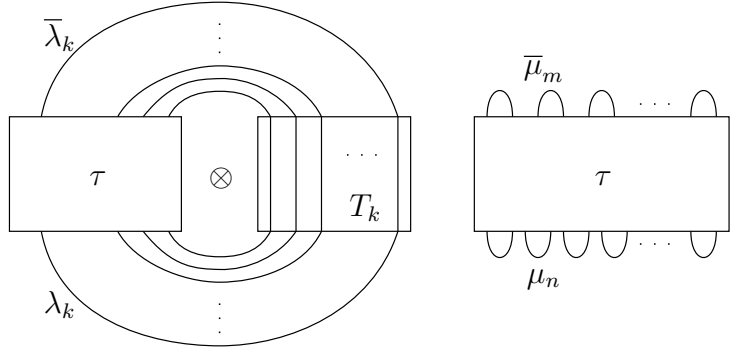
\begin{figure}[h!]
\labellist
\pinlabel $\tau $ at 210 400
\pinlabel $T_k$ at 840 330
\pinlabel $\tau $ at 1400 400
\pinlabel $\lambda _{k}$ at 120 100
\pinlabel $\overline{\lambda }_{k}$ at 120 740
\pinlabel $\mu _n$ at 1260 160 
\pinlabel $\overline{\mu }_m$ at 1260 660 
\pinlabel $\otimes $ at 500 400
\normalsize \hair 2pt
\endlabellist
\begin{center}
\includegraphics[scale=0.16]{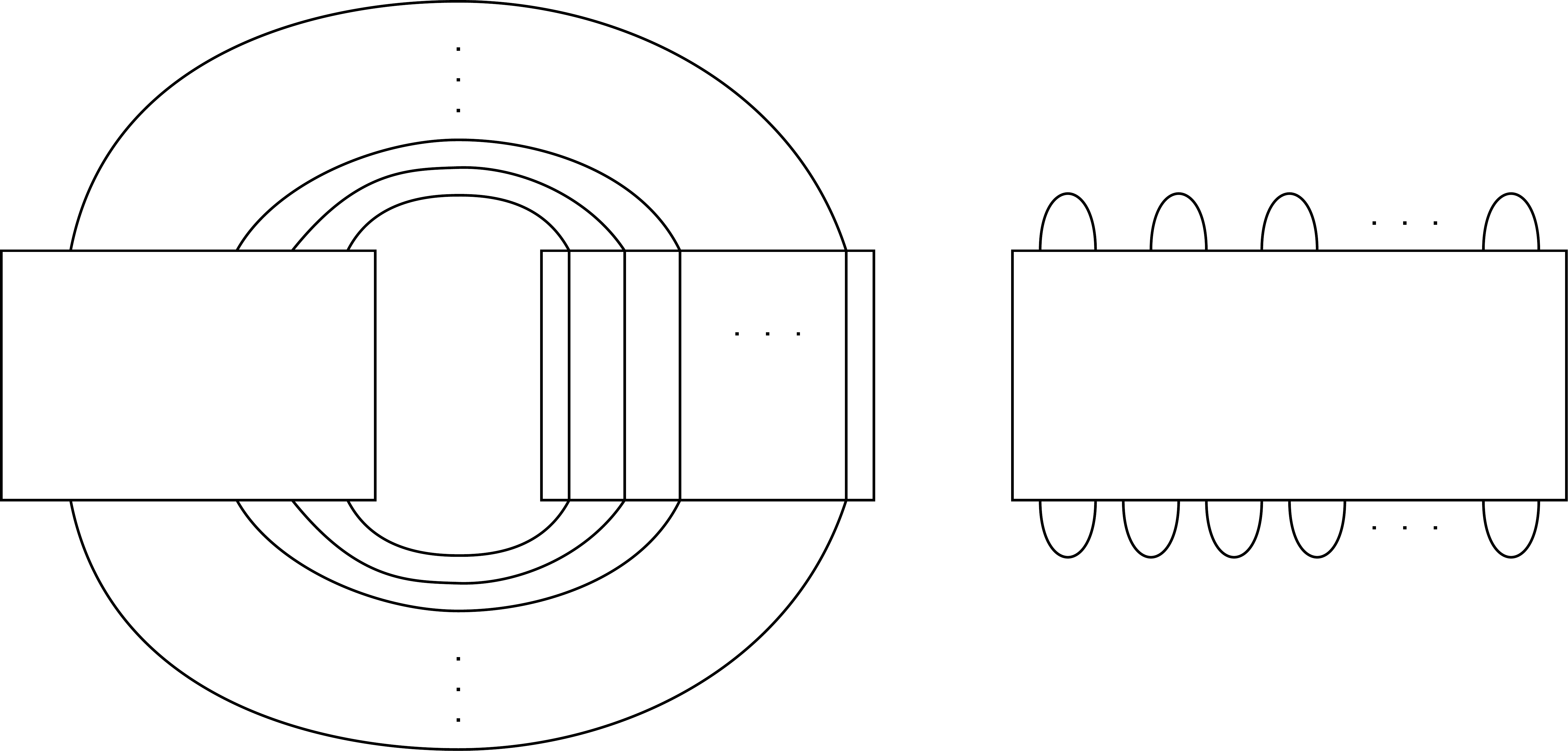}
\caption{Two types of the closing-up operation on tangles: classical closure $\widehat{\tau }$ (left) and plat closure $\widetilde{\tau }$ (right).}
\label{fig_closingup}
\end{center}
\end{figure}

 Analogously, oriented links could be seen as morphisms from the empty set to itself in the oriented tangle category. An oriented version of the classical closure could be defined for each $(\phi,\phi)$-tangle as $\widehat{\tau}=\lambda^{\eta}_k (\tau\otimes T_k^{\psi})\overline{\lambda^{\eta}_k}$, where  $\psi(p_{k-i+1})=-\phi(p_i)$ for $i=1,\ldots,k$, while $\eta$ coincides with $\phi$ on the first $k$ points and coincides with $\psi$ on the last $k$ points. Note that $\overline{\eta}=\eta$.  On the contrary, the plat closure could be defined only for $(\phi,\psi)$-tangles such that  $\phi(p_{2i-1})=-\phi(p_{2i})$ and $\psi(p_{2j-1})=-\psi(p_{2j})$,  for $i=1,\ldots, n$ and $j=1,\ldots,m$ as $\widetilde{\tau}=\mu_n^{\phi}\tau\overline{\mu^{\overline{\psi }}_m}$.

 \end{remark}

 Remark \ref{rem_closing} describes two different ways to ``close'' a tangle. On the contrary, given a link $L$ in $D\times I$, any 2-disk $D_t=D\times \{t\}$, intersecting the link transversely in $2n$ points, decomposes the link into two tangles $\tau_1$ and $\tau_2$, such that $\tau_1$ is a $(0,2n)$-tangle, $\tau_2$ is a $(2n,0)$-tangle and $\tau_1\tau_2=L$. In general, this decomposition depends on the disk $D_t$. Nevertheless,  this idea of splitting can be used, for example, to describe the connected sum of knots, as is done in Subsection \ref{appl2}.

\end{subsection}

\begin{subsection}{Bordered quandle category}\label{sub22}

\begin{definition} \label{def:quandle} A \textit{quandle} is a set $Q$ with a binary operation $\tr :Q\times Q\to Q$, which satisfies the following axioms:
\begin{enumerate}
\item $x\tr x=x$ for every $x\in Q$,
\item the map $\alpha _{y}\colon Q\to Q$, defined by $\alpha _{y}(x)=x\tr y$, is invertible for every $y\in Q$, and 
\item $(x\tr y)\tr z=(x\tr z)\tr (y\tr z)$ for every $x,y,z\in Q$. 
\end{enumerate} 

We set $x\tl y=\alpha_y^{-1}(x)$. This notation implies the identities $(x\tr y)\tl y=x=(x\tl y)\tr y$. 

Moreover, a \textit{quandle homomorphism} is a function $f: (Q_1,\tr_1)\to(Q_2,\tr_2)$ that satisfies $f(x\, \tr_1 y)=f(x)\tr_2 f(y)$ for every $x,y\in Q_{1}$. 

\end{definition}

\begin{example}Any group $(G,\circ )$ is a quandle for the operation $a\tr b:=b^{-1}\circ a\circ  b$. 
\end{example}

The \textit{quandle category} $\mathcal Q$ is the category whose objects are quandles and whose morphisms are quandle homomorphisms. In the remainder of this Section, we will define another category based on quandles: the bordered quandle category, whose objects will be free quandles. 

\begin{definition}
Given a set $X$, the \textit{free quandle} on $X$ is the quandle $(X\times F_G(X),\tr)$, where $F_G(X)$ is the free group over $X$ and the operation is defined as $(x_1,a)\tr(x_2,b)=(x_1,ab^{-1}x_2b)$. The free quandle on $X$ will be denoted by $F(X)$ and the elements of $X$ will be called the \textit{generators} of the free quandle. Clearly, two free quandles $F(X_1)$ and $F(X_2)$   are isomorphic if and only if $X_1$ and $X_2$ have the same cardinality.
\end{definition}

We define the \textit{bordered quandle category} $\qb $ as follows. The objects of $\qb $ are free quandles with finitely many (possibly zero) generators. A free quandle with zero generators is simply the empty set, viewed as an \textit{empty quandle}. For two objects $F_{1}$ and $F_{2}$, a morphism with source $F_{1}$ and target $F_{2}$ is given by a quandle $Q$, together with two quandle homomorphisms $F_1\stackrel{i_{1}}{\rightarrow }Q\stackrel{i_2}{\leftarrow }F_2$. Observe that if either of $F_{1}$, $F_{2}$ is the empty quandle, its inclusion into $Q$ is automatically a quandle homomorphism. 

Given three objects $F_{1},F,F_{2}\in Ob(\qb )$ with morphisms $F_1\stackrel{i_{1}}{\rightarrow }Q_{1}\stackrel{i}{\leftarrow }F$ and $F\stackrel{j}{\rightarrow }Q_{2}\stackrel{j_2}{\leftarrow }F_2$, define their composition by the triple $$F_{1}\stackrel{i_1}{\rightarrow }Q_{1}*_{(F,i,j)}Q_{2}\stackrel{j_2}{\leftarrow }F_2\;,$$ where $Q_{1}*_{(F,i,j)}Q_{2}$ denotes the amalgamated product of $Q_{1}$ and $Q_{2}$ with respect to the triple $(F,i,j)$, that is the pushout in the quandle category. It is easy to see this is indeed a category: amalgamation of quandles as a pushout is associative, while any object $F\in Ob(\qb )$ admits the identity morphism $F\stackrel{\textrm{id}}{\rightarrow }F\stackrel{\textrm{id}}{\leftarrow }F$ that satisfies the required property. 

\end{subsection}
\end{section}

\begin{section}{Fundamental quandle as a functor}
\label{sec3}

In this Section, we recall the definition of the fundamental quandle of a codimension 2 submanifold of a topological manifold. Moreover, we show that by using the fundamental quandle, it is possible to define a functor from the oriented tangle category $\mathcal{T}^{+}$ to the bordered quandle category $\qb$. 

\begin{definition}[\cite{FR}] Let $L\subset M$ be a codimension 2 submanifold of a connected manifold $M$. If $L$ is not closed, assume that $\partial L$ is contained in a connected component of $\partial M$.  Also assume that $L$ is transversely oriented in $M$, denote by $N_{L}$ the normal disk bundle and by $E_{L}=cl(M-N_{L})$ its exterior. Set 
$$\Gamma _{L}= \{\textrm{homotopy classes of paths in $E_{L}$ from a point on $\partial N_{ L}$ to a basepoint}\};$$ 
this will be the underlying set of the fundamental quandle. During homotopy, the basepoint is kept fixed, while the initial point of the path is free to move in $\partial N_{L}$.  For any point $p\in \partial N_{L}$, denote by $m_{p}$ the loop in $\partial N_{L}$ based at $p$, which follows around the meridian in the positive direction. Let two elements $a,b\in \Gamma _{L}$ be represented by the paths $\alpha $ and $\beta $ respectively. The \textit{fundamental quandle} $Q(L)$ of $L$ is the set $\Gamma _{L}$ with operation $$a\tr b := [\alpha \cdot \overline{\beta }\cdot m_{\beta (0)}\cdot \beta ]\;,$$ where $\overline{\beta }$ is the reverse path to $\beta $ and $\cdot $ denotes the composition of paths.

If $L_{0}\subset M$ is an empty submanifold of a connected manifold $M$, then we define its \textit{fundamental quandle} $Q(L_{0})$ to be the empty quandle. 
\end{definition}

\begin{example}\label{ex3} Given $P_n^{\phi}\in Ob(\mathcal T^{+})$, the underlying  set $\{p_1,\ldots, p_n\}$  is a codimension 2 submanifold of the disk $D$. We assume the orientation of the disk is induced by the standard orientation of $\mathbb R^2$, so that its  boundary is oriented counterclockwise, and that the meridian $m_{p_i}$ of a point $p_i$  is oriented counterclockwise (resp. clockwise) if $\phi(p_i)=+1$ (resp. $\phi(p_i)=-1$). Let $a_i$ be  the homotopy class of   the path $\alpha_i$ depicted in Figure \ref{fig_disk}, and let $\mu_i$ be the  homotopy class of  $m_{p_i}$ for $i=1,\ldots,n$.  The elements of $Q(P_n^{\phi})$ have the form $a_iw(b_1,\ldots,b_n)$,  where $w(b_1,\ldots,b_n)$ is a word in the letters $b_1,\ldots,b_n$ with  $b_i=a_i^{-1}\mu_ia_i$. It follows that $Q(P_n^{\phi})=F(a_1,\ldots,a_n)$.
\end{example}

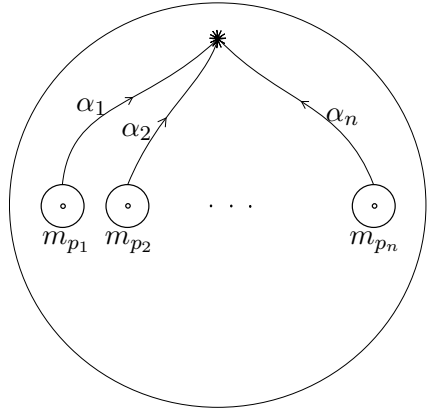
\begin{figure}[h!]
\labellist
\small \hair 2pt
\pinlabel $\alpha_{1}$ at 165 600
\pinlabel $\alpha_{2}$ at 255 550
\pinlabel $\alpha_{n}$ at 665 580
\pinlabel $m_{p_{1}}$ at 115 330
\pinlabel $m_{p_{2}}$ at 240 330
\pinlabel $m_{p_{n}}$ at 730 330
\endlabellist
\begin{center}
\includegraphics[scale=0.20]{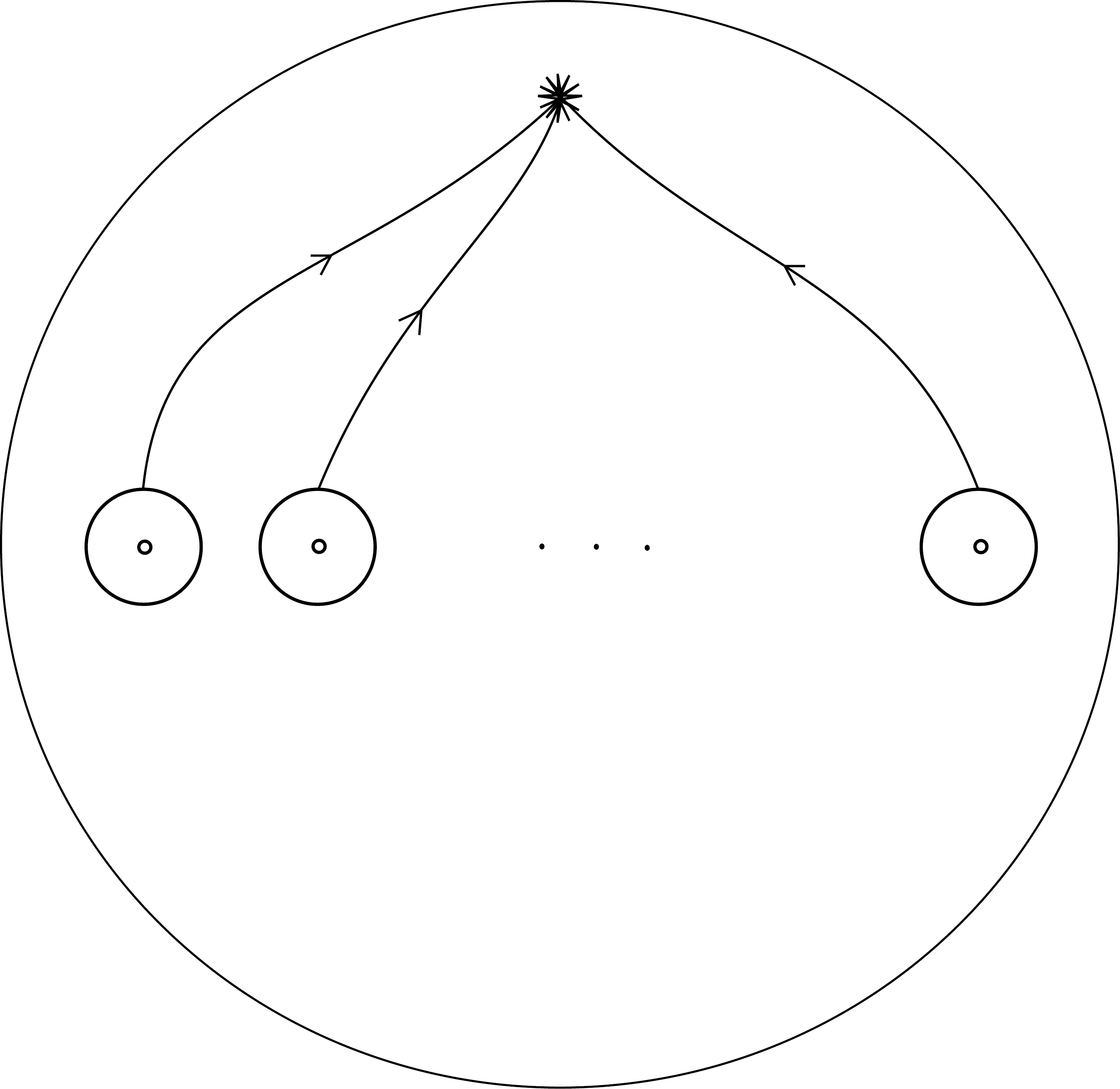}
\caption{Generators of $Q(P^{\phi}_{n})$.}
\label{fig_disk}
\end{center}
\end{figure}

\begin{remark}
\label{rem_sub}
If $L$ has nonempty boundary, the inclusion $\iota:\partial L\hookrightarrow L$ induces a quandle homomorphism $i:Q(\partial L)\to Q(L)$. Indeed, denoting by $\partial M_{0}$ the connected component of $\partial M$ that contains $\partial L$, any element of $Q(\partial L)$ is the homotopy class of a path in $cl(\partial M_{0}-N_{\partial L})$ from $\partial N_{\partial L}$ to the basepoint. The homotopy class of this path in $cl(M-N_{L })$ is clearly an element of $Q(L)$. The inclusion induces a homomorphism if $i(a\vartriangleright  b)=[\iota(\alpha \cdot \overline{\beta }\cdot m_{\beta (0)}\cdot \beta)]$ is equal to $i(a)\vartriangleright i(b)=[\iota(\alpha) \cdot \overline{\iota(\beta )}\cdot m_{\iota(\beta (0))}\cdot \iota(\beta) ]$. This holds if $\iota(m_{\beta (0)})=m_{\iota(\beta (0))}$, which means that a positive meridian in $\partial N_{\partial L}$ is still positive  in $\partial N_L$. For this to be true, we need to assume that the orientation of $\partial L$ is induced by that of $L$ or, equivalently, that $\iota$ is orientation preserving. 
\end{remark}
Now we are ready to define a functor from the oriented tangle category to the bordered quandle category.
 
 \begin{definition}  The functor $BQ\colon \mathcal{T}^+\to \qb$ from the oriented  tangle category to the bordered quandle category is defined as follows: 
 
  For an object $P_n^{\phi}\in Ob(\mathcal{T}^+)$ with $n\geq 1$, we set $BQ(P_n^{\phi})=Q(P_n^{\phi})$, which is a free quandle with $n$ generators by the Example \ref{ex3}. If $P_0\in Ob(\mathcal{T}^+)$ is the empty object, then $BQ(P_0)$ is the empty quandle. 

Given a morphism $\tau \in Hom(P_n^{\phi},P_m^{\psi})$, the inclusions $P_n^{\phi}\hookrightarrow \tau$ and $P_m^{\psi} \hookrightarrow \tau$ induce quandle homomorphisms $i_{\tau}^-\colon Q(\mathcal{P}^{\phi}_n)\to Q(\tau )$ and  $i_{\tau}^+\colon Q(P_m^{\psi})\to Q(\tau )$ by Remark \ref{rem_sub}. We define $BQ(\tau)$ to be the morphism of $\qb$, given by the triple $Q(P_n^{\phi})\stackrel{i_{\tau}^-}{\rightarrow}Q(\tau )\stackrel{i_{\tau}^+}{\leftarrow}Q(P_{m}^{\psi})$. We use the convention that if either $m$ or $n$ is equal to zero, then the corresponding homomorphism is omitted from the triple. We call $BQ$ the \textit{fundamental quandle functor}.
\end{definition}
 
In order to justify this definition, we need to prove that $BQ$ actually is a functor. 

\begin{Theorem} With the above notations, $BQ\colon \mathcal{T}^+\to \qb$ is a functor. 
\end{Theorem}
\begin{proof}

For any object $P_n^{\phi}\in Ob(\mathcal{T}^+)$, its identity morphism $\textup{id}_{P_n^{\phi}}$ is the trivial tangle $T_n^{\phi}$ (see the left picture of Figure \ref{fig_cupcap2}). It is straightforward to check that the fundamental quandle of this tangle is the free quandle with the same set of generators as $Q(P_n^{\phi})$ (see Example \ref{ex3}),  so $BQ(T_n^{\phi})$ is given by the triple  $Q(P_n^{\phi})\stackrel{\textrm{id}}{\rightarrow}Q(T_n^{\phi})\stackrel{\textrm{id}}{\leftarrow }Q(P_n^{\phi})$, thus $BQ(\textup{id}_{P_n^{\phi}})=\textrm{id}_{BQ(P_n^{\phi})}$. 

Let $\tau _{1}\in Hom(P_n^{\phi},P_m^{\psi})$ be a $({\phi},{\psi})$-tangle and let $\tau _{2}\in Hom(P_m^{\psi},P_p^{\rho})$ be a $({\psi},{\rho})$-tangle such that $\partial ^{+}\tau _{1}=\partial ^{-}\tau _{2}$. Let $BQ(\tau _{1})$ and $BQ(\tau _{2})$ be given by triples $Q(P_n^{\phi})\stackrel{i_{\tau_1}^-}{\rightarrow}Q(\tau _1)\stackrel{i_{\tau_1}^+}{\leftarrow}Q(P_{m}^{\psi})$ and $Q(P_m^{\psi})\stackrel{i_{\tau_2}^-}{\rightarrow}Q(\tau _2)\stackrel{i_{\tau_2}^+}{\leftarrow}Q(P_{p}^{\rho})$. Then $BQ(\tau_1\tau _2)$ is given by the triple 

$$Q(P_n^{\phi})\stackrel{i_{\tau_1}^-}{\rightarrow}Q(\tau _1)*_{(Q(P_m^{\psi}),i_{\tau_1}^+,i_{\tau_2}^-)}Q(\tau _2)\stackrel{i_{\tau_2}^+}{\leftarrow}Q(P_{p}^{\rho})\;.$$ Indeed, the composite tangle $\tau _1 \tau _2$ is the $({\psi},{\rho})$-tangle, obtained by stacking the tangle $\tau _2$ over $\tau _1$ and identifying the common boundary $\partial ^{+}\tau _1=\partial ^{-}\tau _2$. Every element of $Q(\tau _1 \tau _2)$ is represented by a path from $\partial N_{\tau _1}$ or $\partial N_{\tau _2}$ inside $\partial N_{\tau _1 \tau _2}$ to the basepoint. So, for $i=1,2$ we can define quandle homomorphisms $j_i:Q(\tau_i)\to Q(\tau _1 \tau _2)$ such that  $j_1(i_{\tau_1}^+(a))=j_2(i_{\tau_2}^-(a))$ for every $a\in Q(P_{m}^{\psi})$. It follows that $Q(\tau _{1}\tau _2)=Q(\tau _1)*_{(Q(P_m^{\psi}),i_{\tau_1}^+,i_{\tau_2}^-)}Q(\tau _2)$ and the boundary inclusions give quandle homomorphisms $i_{\tau_1}^-\colon Q(P_n^{\phi})\to Q(\tau _1)*_{(Q(P_m^{\psi}),i_{\tau_1}^+,i_{\tau_2}^-)}Q(\tau _2)$ and $i_{\tau_2}^+\colon Q(P_{p}^{\rho})\to Q(\tau _1)*_{(Q(P_m^{\psi}),i_{\tau_1}^+,i_{\tau_2}^-)}Q(\tau _2)$, which concludes the proof.
\end{proof}

\begin{example}\label{ex4}
Consider the cup and plat tangles depicted in Figures \ref{fig_cupcap1} and \ref{fig_cupcap2}. The fundamental quandles $Q(\lambda_k^{\phi})$ and $Q(\mu_k^{\psi})$ are both free quandles on $k$ generators $x_1,\ldots,x_k$, where $x_i$ is the homotopy class of a path from the basepoint to the tubular neighborhood of $\chi_i$. Moreover, the functor $BQ$ sends:
\begin{itemize}
 \item  the cup tangle $\lambda_k^{\phi}$ into $Q(\lambda_k ^{\phi})\stackrel{i_k^+}{\leftarrow}Q(P_{2k}^{\phi})$, where $i_k^+:F(a_1,\ldots,a_{2k})\to F(x_1,\ldots,x_k)$ is the homomorphism, given by $i_k^+(a_h)=i_k^+(a_{2k-h+1})=x_h$ for $h=1,\ldots,k$;
 \item the plat tangle $\mu_k^{\psi}$ into $Q(\mu _k^{\psi} )\stackrel{j_k^+}{\leftarrow}Q(P_{2k}^{\psi})$, where $j_k^+:F(a_1,\ldots,a_{2k})\to F(x_1,\ldots,x_k)$ is the homomorphism, given by $j_k^+(a_{2h-1})=j_k^+(a_{2h})=x_h$ for $h=1,\ldots,k$.
\end{itemize}
\end{example}

\begin{remark}\label{rem_tensor}
 The functor $BQ$  behaves nicely with respect to the tensor product of tangles. Given $\tau _{1}\in Hom(P_{n_{1}}^{\phi_1},P_{m_1}^{\psi_1})$ and $\tau _{2}\in Hom(P_{n_2}^{\psi_2},P_{m_2}^{\phi_2})$, the functor sends their product $\tau _{1}\otimes\tau_{2}$ into $$Q(P_{n_1+n_2}^{(\phi_1,{\phi_2})})\stackrel{i_{\tau_1}^-\ast i_{\tau_2}^-}{\longrightarrow}Q(\tau _1)\ast Q(\tau_2)\stackrel{i_{\tau_1}^+\ast i_{\tau_2}^+}{\longleftarrow}Q(P_{m_1+m_2}^{(\psi_1,{\psi_2})})\;,$$ where $BQ(\tau_j):\ Q(P_{n_j}^{\phi_j})\stackrel{i_{\tau_j}^-}{\rightarrow}Q(\tau_j )\stackrel{i_{\tau_j}^+}{\leftarrow}Q(P_{m_j}^{\psi_j})$ for $j=1,2$.
 
 Moreover, if $BQ(\tau):\ Q(P_{n}^{\phi})\stackrel{i_{\tau}^-}{\rightarrow}Q(\tau)\stackrel{i_{\tau}^+}{\leftarrow}Q(P_{m}^{\psi})$, then  $BQ(\overline{\tau}):\ Q(P_{m}^{\overline{\psi}})\stackrel{i_{\tau}^+}{\rightarrow}Q(\overline{\tau})\stackrel{i_{\tau}^-}{\leftarrow}Q(P_{n}^{\overline{\phi}})$ and clearly $Q(\tau)\cong Q(\overline{\tau})$.
\end{remark}

\begin{remark}\label{rem_fr}
As we mentioned in Section \ref{sec2}, braids on $k$ strands are isomorphisms of $P_k$ in the tangle category. While the fundamental quandle of a braid is always free, our functor $BQ$ sends a braid $\beta $ into a triple $Q(P_k)\stackrel{i^{-}}{\rightarrow }Q(\beta )\stackrel{i^{+}}{\leftarrow}Q(P_k)$, where $i^-$ and $i^{+}$ are both quandle isomorphisms. Therefore the composition $(i^{+})^{-1}i^{-}=\beta _{*}\colon Q(P_{k})\to Q(P_k)$ is an automorphism of the free quandle. In \cite[pages 390--391]{FR} a similar remark was made in the category of racks (i.e. algebraic structures, satisfying properties (2) and (3) of Definition \ref{def:quandle}), and the authors proved that the map $\beta \mapsto \beta _{*}$ defines a faithful representation of the braid group in the automorphism group of the free rack. 
\end{remark}

As is the case with groups, every quandle admits a quandle presentation. Before recalling its definition, we need to fix some notation. Any free quandle  $F(X)$ is characterized by the following universal property: for any function $f$ from $X$ to a quandle $Q$, there exists a unique quandle homomorphism $\overline{f}\colon F(X) \to Q$, such that  $\overline{f} \iota =f$, where $\iota:X\to F(X)$ is the inclusion, given by $\iota (x)=(x,1)$.

\begin{definition} 
Let $Q$ be a quandle. A \textit{presentation} $\langle X\mid R\rangle$ for $Q$ is the data of two sets $X$ and $R\subset F(X)\times F(X)$, and a function $g:X\to Q$ that satisfies: 
\begin{itemize}
\item[(a)] $\overline{g}\times \overline {g}(R)\subset \Delta_Q$, where $\Delta_Q$ is the diagonal set in $Q\times Q$, and 
\item[(b)] for any quandle $Q'$ and any map $f: X\to Q'$ with $\overline{f}\times \overline{f}(R)\subset \Delta_{Q'}$, there exists a unique quandle homomorphism $\widetilde f:Q\to Q'$ such that $f=\widetilde fg$.
\end{itemize}
\end{definition}

To avoid cumbersome notations, we will usually replace a relation $(r,s)\in R$ by an equality $r=s$, and denote an  element $\iota(x)\in F(X)$ simply  by $x$. Moreover, we will write $\langle x_1,\ldots, x_m  \mid r_1,\ldots, r_n\rangle$  instead of $\langle \{x_1,\ldots, x_m\}\mid \{r_1,\ldots, r_n\}\rangle$. 

\begin{corollary}\label{cor2}
Given a $(\phi,\psi)$-tangle $\tau _1$ and a $(\psi,\rho)$-tangle $\tau_2$, let their fundamental quandles be given by presentations $Q(\tau_1)=\langle X_1\mid R_1\rangle$ and $Q(\tau_2)=\langle X_2\mid R_2\rangle$. \\If $BQ(\tau_1)\colon Q(P_n^{\phi})\stackrel{i_{\tau_1}^-}{\rightarrow}Q(\tau _1)\stackrel{i_{\tau_1}^+}{\leftarrow}Q(P_{m}^{\psi})$ and $BQ(\tau_2)\colon Q(P_m^{\psi})\stackrel{i_{\tau_2}^-}{\rightarrow}Q(\tau _2)\stackrel{i_{\tau_2}^+}{\leftarrow}Q(P_{p}^{\rho})$, then the fundamental quandle of $\tau _1\tau _2$ is given by the presentation
 $$ Q(\tau_1\tau_2)=\langle X_1\cup X_2\mid R_1\cup R_2\cup \{i_{\tau_1}^+(a_j)=i_{\tau_2}^-(a_j),\ j=1,\ldots m\}\rangle,$$ where  $Q(P_{m}^{\psi})=F( a_{1},\ldots ,a_{m})$. 
\end{corollary}

In \cite{FR} the authors describe how to find a presentation for the fundamental quandle of a link in $S^3$. \label{presentation} The same procedure easily generalizes to a tangle in $D\times I$. Recall that a \textit{tangle diagram} is the image of a regular  projection of a tangle onto $[-1,1]\times I$, where $[-1,1]=D\cap \mathbb R$, together with the information about the overcrossings and undercrossings. If $D_{\tau}$ is a  diagram of $\tau$, denote by $x_1,\ldots, x_n$ the arcs of $D_{\tau }$, let $m$ be the number of crossing points and fix an orientation over $D_{\tau}$ (induced by that of $\tau$ if $\tau$ is oriented).   A presentation for $Q(\tau)$ is then given by: $$Q(\tau)=\left \langle x_1,\ldots, x_n\mid r_1,\ldots, r_m\right \rangle ,$$ where 
\begin{xalignat}{1}\label{qrelation}
& r_j\colon x_i=x_h\tr x_l
\end{xalignat} if the arcs of $D_{\tau }$ at the $j$-th crossing are labeled as depicted in the left part of Figure \ref{fig_ntangle}.

\end{section}

\begin{section}{Fundamental quandle of links as closures of tangles}
\label{sec4}
In this Section, we focus on the study of links. First we analyze how the functor introduced in Section \ref{sec3} can be used to compute the fundamental quandle of the closure of a tangle. Then we concentrate on some particular classes of links: periodic links, composite knots and satellite knots. \\

Let $\tau$ be a $(\phi,\phi)$-tangle and let  $BQ(\tau): \ Q(P_k^{\phi})\stackrel{i_{\tau}^-}{\rightarrow}Q(\tau )\stackrel{i^+_{\tau}}{\leftarrow}Q(P^{\phi}_{k})$. In Remark \ref{rem_closing}, we defined the classical closure of $\tau$ as 
  $\widehat{\tau}=\lambda^{\eta}_k (\tau\otimes T_k^{\psi})\overline{\lambda^{\eta}_k}$, where  $\psi(p_{k-i+1})=-\phi(p_i)$ for $i=1,\ldots,k$ and $\eta$ coincides with $\phi$ on the first $k$ points and with $\psi$ on the last $k$ points. By Remark \ref{rem_tensor}, we obtain  
$$BQ(\widehat{\tau}):\ Q(\lambda _k^{\eta})*_{(Q(P_{2k}^{\eta}),i_k^+,i_{\tau}^-\ast \textup{id}_{Q(P_k^{\psi})})}(Q(\tau)*Q(T^{\psi}_k))*_{(Q(P_{2k}^{\eta}),i_{\tau}^+\ast \textup{id}_{Q(P_k^{\psi})},i_k^+)}Q(\overline{\lambda _k^{\eta}}),$$
where $i_k^+$ is the map described in Example \ref{ex4}. \\

Using presentations $Q(\tau)=\langle Y\mid R\rangle$, $Q(T_k^{\psi})=F(z_1,\ldots,z_k)=F(Z)$, $Q(\lambda_k^{\eta})=F(x_1,\ldots,x_k)=F(X)$, $Q(\overline{\lambda_k^{\eta}})=F(\overline{x}_1,\ldots,\overline{x}_k)=F(\overline{X})$ and applying Corollary \ref{cor2} together with computations in Example \ref{ex4}, we obtain
\begin{align*}Q(\widehat{\tau})=\langle& Y\cup Z\cup X \cup \overline{X} \mid R\cup \{i_{\tau}^-(a_h)=x_h=z_{2k-h+1},  i_{\tau}^+(a_h)=\overline{x}_h=z_{2k-h+1},\ h=1,\ldots, k \}  \rangle,
\end{align*}
with $Q(P_{2k}^{\eta})=F(a_1,\ldots,a_{2k})$, that easily simplifies to 
\begin{equation}
\label{eq:ordinary}
 Q(\widehat{\tau})=\langle Y\mid R\cup \{i_{\tau}^-(a_h)=i_{\tau}^+(a_h), \ h=1,\ldots, k\}  \rangle.
\end{equation}

Now let $\tau $ be a $(\phi,\psi)$-tangle such that $\phi(p_{2i-1})=-\phi(p_{2i})$ and $\psi(p_{2j-1})=-\psi(p_{2j})$ for $i=1,\ldots, n$ and $j=1,\ldots,m$. A similar computation for the plat closure $\widetilde{\tau}=\mu_n^{\phi}\tau\overline{\mu^{\overline{\psi }}_m}$ gives 
$$BQ(\widetilde{\tau}):\  Q(\mu_n^{\phi })*_{(Q(P_{2n}^{\phi}),j_n^+,i_{\tau}^-)}Q(\tau)*_{(Q(P_{2m}^{\psi}),i_{\tau}^+,j_m^+)}Q(\overline{\mu _m^{\overline{\psi }}}),$$
where $BQ(\tau):\  Q(P_{2n}^{\phi})\stackrel{i_{\tau}^-}{\rightarrow}Q(\tau )\stackrel{i^+_{\tau}}{\leftarrow}Q(P_{2m}^{\psi})$ and $j_n^+,j_m^+$ are the maps described in Example \ref{ex4}. Moreover, given presentations $Q(\tau)=\langle Y\mid R\rangle$, $Q(\mu_n^{\phi})=F(x_1,\ldots,x_n)=F(X)$ and $Q(\overline{\mu_m^{\overline{\psi}}})=F(\overline{x}_1,\ldots,\overline{x}_m)=F(\overline{X})$, we obtain 
\begin{align*}Q(\widetilde{\tau})=\langle& Y\cup X \cup \overline{X} \mid R\cup \{i_{\tau}^-(a_{2h-1})=x_h=i_{\tau}^-(a_{2h}),\\
& \ i_{\tau}^+(b_{2j-1})=\overline{x}_h=i_{\tau}^+(b_{2j}),\ h=1,\ldots, n, \ j=1,\ldots, m  \}  \rangle,
\end{align*}
where $Q(P_{2n}^{\phi})=F(a_1,\ldots,a_{2n})$ and $Q(P_{2m}^{\psi})=F(b_1,\ldots,b_{2m})$, that easily simplifies to
\begin{align}
\label{eq:plat}
Q(\widetilde{\tau})=\langle & Y\mid R\cup \{i_{\tau}^-(a_{2h-1})=i_{\tau}^-(a_{2h}), \ i_{\tau}^+(b_{2j-1})=i_{\tau}^+(b_{2j}),\ \ h=1 \ldots, n, \ j=1,\ldots, m  \}  \rangle.
\end{align}
In the remainder of this Section, we will apply presentations \eqref{eq:ordinary} and \eqref{eq:plat} to the computation of the fundamental quandle of periodic links, composite knots and satellite knots.

\begin{subsection}{Periodic links}
\label{appl1}
 
  A link $L$ in $S^3$ is called \textit{periodic} with period $p$ or $p$-periodic, if there is an orientation preserving homeomorphism $\theta_p: S^3\to S^3$ of order $p$  with a set of fixed points $h\cong S^1$ disjoint from $L$, which maps $L$ to itself. Periodic links are important for the knot theory in lens spaces, since they are lifts of links in lens spaces under the universal covering $S^3\to L(p,q)$ (see for example \cite{CM}). \\
  
 Consider $S^3$ as $\mathbb R^3\cup\{\infty\}$ with $\infty\in h$. A $p$-periodic link admits a regular diagram $D_L$, contained in a plane $\pi\subset \mathbb R^3$ orthogonal to $h$, such that the restriction of $\theta_p$ to $\pi$ is a rotation of order $p$ around the point $O\in \pi\cap h$ and  fixes the diagram  $D_L$ (see \cite[page 267]{BZ}). Let $A$ be the convex angle, delimited by a ray $r$ exiting from $O$, and its image $\theta_p(r)$: the set $D_L\cap A$ can clearly be seen as a diagram of a $(k,k)$-tangle $\tau$, where $k$ is the cardinality of $D_L\cap r$. So $L$ can be expressed as $\widehat{\tau^p}$ (see Remark \ref{rem_closing} and Figure \ref{fig:periodic}). On the contrary, given any $(k,k)$-tangle $\tau$, the link $\widehat{\tau^p}$ is a $p$-periodic link. 
 
\begin{figure}[h!]
\labellist
\normalsize \hair 2pt
\pinlabel $O$ at 480 -25
\pinlabel $A$ at 505 85
\endlabellist
\begin{center}
\includegraphics[scale=0.23]{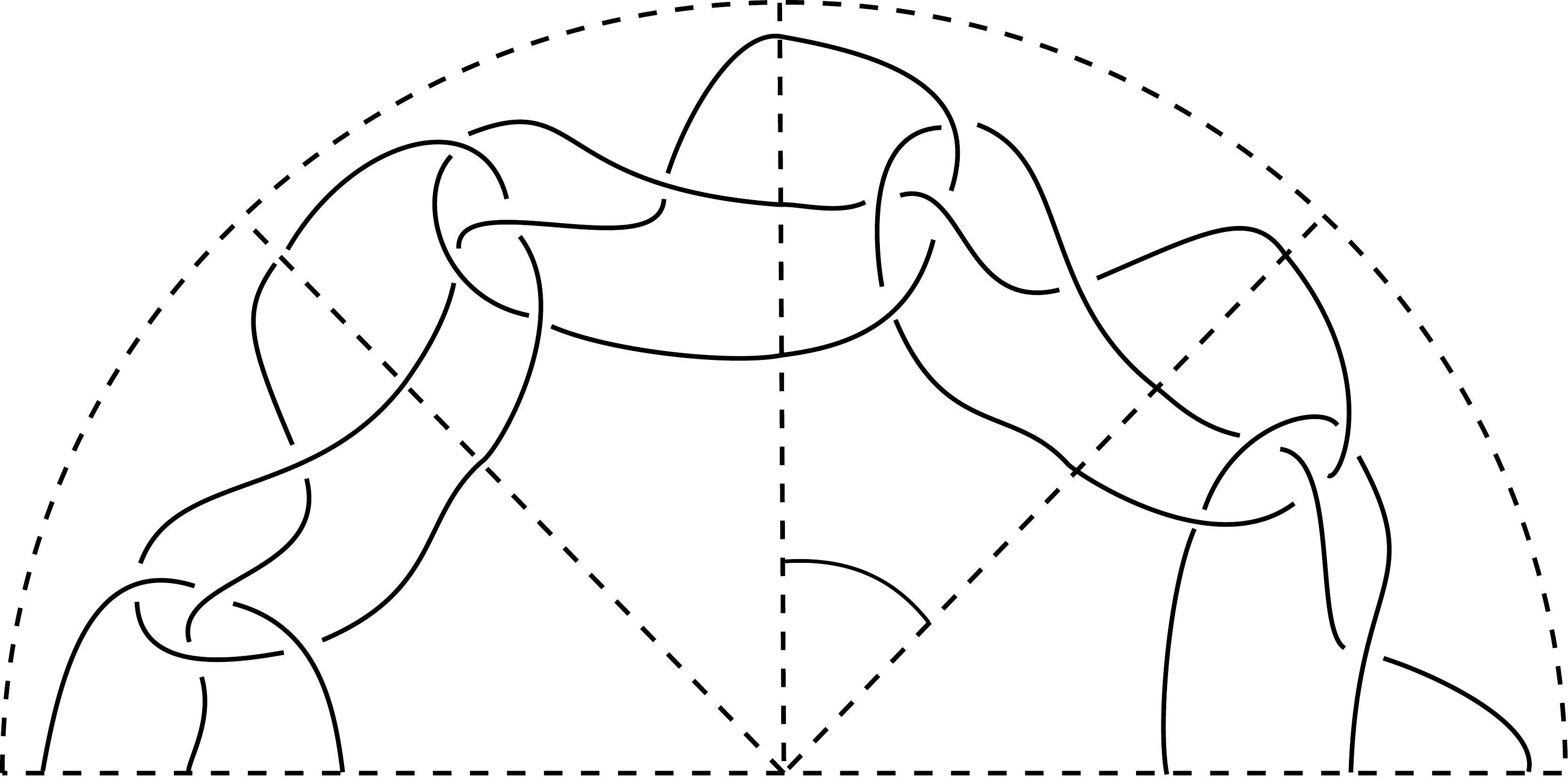}
\caption{Part of a diagram of a periodic knot, which may be expressed as $\widehat{\tau ^{8}}$ for a $(3,3)$-tangle $\tau $. }
\label{fig:periodic}
\end{center}
\end{figure}

 If $\tau$ is  a $(\phi,\phi)$-tangle with presentation $\langle y_{i,1}\mid  r_{j,1}\ \ i=1,\ldots ,m,\ j=1,\ldots, n\rangle$ and such that  $BQ(\tau):\ Q(P_k^{\phi})\stackrel{i_{\tau}^-}{\rightarrow}Q(\tau )\stackrel{i_{\tau}^+}{\leftarrow}Q(P_{k}^{\phi})$, then it follows from Corollary \ref{cor2} and presentation \eqref{eq:ordinary} that the fundamental quandle of the $p$-periodic link $\widehat{\tau^p}$  has the following presentation:
 \begin{align*}
 \langle y_{i,z} \mid &\  \psi ^z(r_{j,z}), \psi ^{z-1}(i_{\tau}^+(a_h))=\psi ^{z}(i_{\tau}^-(a_h)), 
 \ i=1,\ldots ,m,\ j=1,\ldots,    n,\\ &  z=1,\ldots, p,\ h=1,\ldots, k \rangle,  
 \end{align*}
where $Q(P_{k}^{\phi})=F(a_1,\ldots, a_k)$ and  the function $\psi$ acts on a word   shifting  by one mod $p$ the second index of each letter appearing in the word.

\begin{figure}[h!]
\labellist
\normalsize \hair 2pt
\pinlabel $y_1$ at 740 90
\pinlabel $y_2$ at 1015 90
\pinlabel $y_3$ at 850 300
\pinlabel $y_4$ at 750 315
\pinlabel $y_5$ at 1000 315
\pinlabel $a_1$ at 700 30
\pinlabel $a_2$ at 1060 30 
\pinlabel $a_1$ at 700 380
\pinlabel $a_2$ at 1060 380 
\endlabellist
\begin{center}
\includegraphics[scale=0.27]{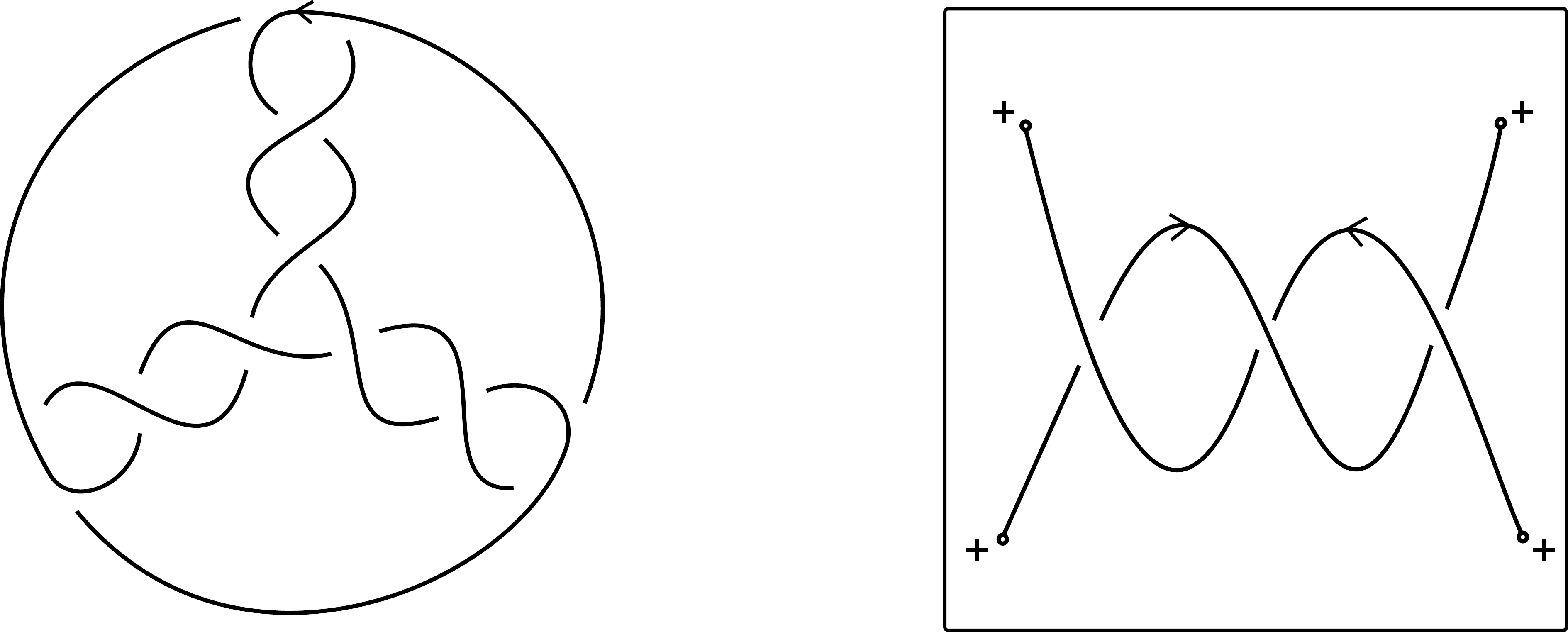}
\caption{Diagram of the pretzel  knot $P(3,3,3)$ (left) and its corresponding tangle (right).}
\label{fig:pretzel}
\end{center}
\end{figure}
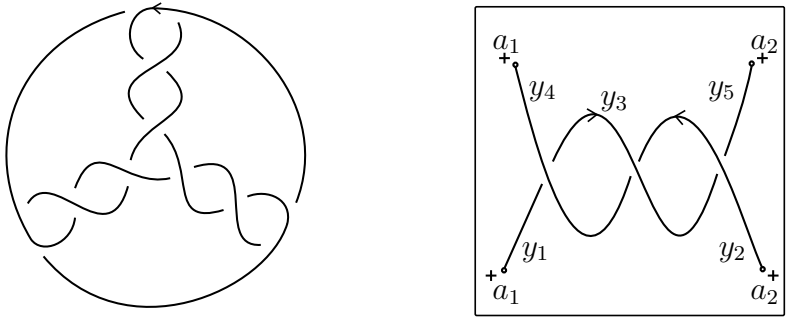

 \begin{example}[Fundamental quandle of the pretzel knot $P(3,3,3)$] Consider the knot $9_{35}$, depicted in the left part of Figure \ref{fig:pretzel}, that is also known as the pretzel knot $P(3,3,3)$. As the diagram shows, it is a 3-periodic knot and can be described as $\widehat{\tau ^{3}}$, where $\tau$ is the  $(2,2)$-tangle depicted in the right part of the figure. Using the procedure, described on p.\pageref{presentation}, we can easily obtain $Q(\tau)=\langle y_1,y_2,y_3,y_4,y_5\, |\, y_3\tr y_4=y_1, y_4\tr y_3=y_2, y_5\tr y_2=y_3\rangle $. The functor $BQ$ maps $\tau $ into $BQ(\tau )\colon \quad F(a_1,a_2)\stackrel{i_{\tau}^{-}}{\rightarrow }Q(\tau )\stackrel{i_{\tau}^{+}}{\leftarrow }F(a_1,a_2)$, where $i_{\tau}^{-}(a_{1})=y_{1}$, $i_{\tau}^{-}(a_2)=y_2$, $i_{\tau}^{+}(a_1)=y_4$ and $i_{\tau}^{+}(a_2)=y_5$. So the fundamental quandle of $P(3,3,3)$ has a presentation 
 \begin{align*}
  \langle y_{i,z}\mid &\ 
  y_{3,z}\tr y_{4,z}=y_{1,z}, y_{4,z}\tr y_{3,z}=y_{2,z}, y_{5,z}\tr y_{2,z}=y_{3,z},  y_{4,z}=y_{1,z+1},\\ &\ y_{5,z}=y_{2,z+1}, z=1,2,3, \ i=1,\ldots ,5\rangle,
 \end{align*} that simplifies to
  \begin{equation*}
  \langle y_{1,z},y_{2,z}\mid (y_{2,z+1}\tr y_{2,z})\tr y_{1,z+1}=y_{1,z}, \ y_{1,z+1}\tr (y_{2,z+1}\tr y_{2,z})=y_{2,z}, z=1,2,3\rangle,
 \end{equation*} where all the second indices are counted mod $3$.
 \end{example}

\end{subsection}

\begin{subsection}{Composite knots}
\label{appl2}

Let $K_1$ and $K_2$ be two nontrivial knots in the solid cylinder $D\times I$. Choose decompositions $K_1=\widehat{\tau_1}$ and $K_2=\widehat{\tau_2}$, where $\tau_1,\tau_2$ are $(1,1)$-tangles.  The connected sum $K_1\sharp K_2$ of $K_1$ and $K_2$ can be expressed as $\lambda_1\left(\tau_1\otimes\tau _2\right)\overline{\lambda}_1$. While tangles $\tau _1$ and $\tau _2$ clearly depend on the chosen decomposition, the connected sum $K_{1}\sharp K_2$ is uniquely defined by $K_1$ and $K_2$. 

\begin{figure}[h!]
\labellist
\normalsize \hair 2pt
\pinlabel $\overline{\lambda }_{1}$ at 420 -30
\pinlabel $\tau _1$ at 130 310
\pinlabel $\tau _2$ at 700 310
\pinlabel $\lambda _1$ at 420 560
\pinlabel $x_1$ at 130 380
\pinlabel $y_1$ at 700 380
\pinlabel $X$ at 130 250
\pinlabel $Y$ at 700 250 
\pinlabel $K_1$ at 400 400
\pinlabel $K_2$ at 970 400
\endlabellist
\begin{center}
\includegraphics[scale=0.20]{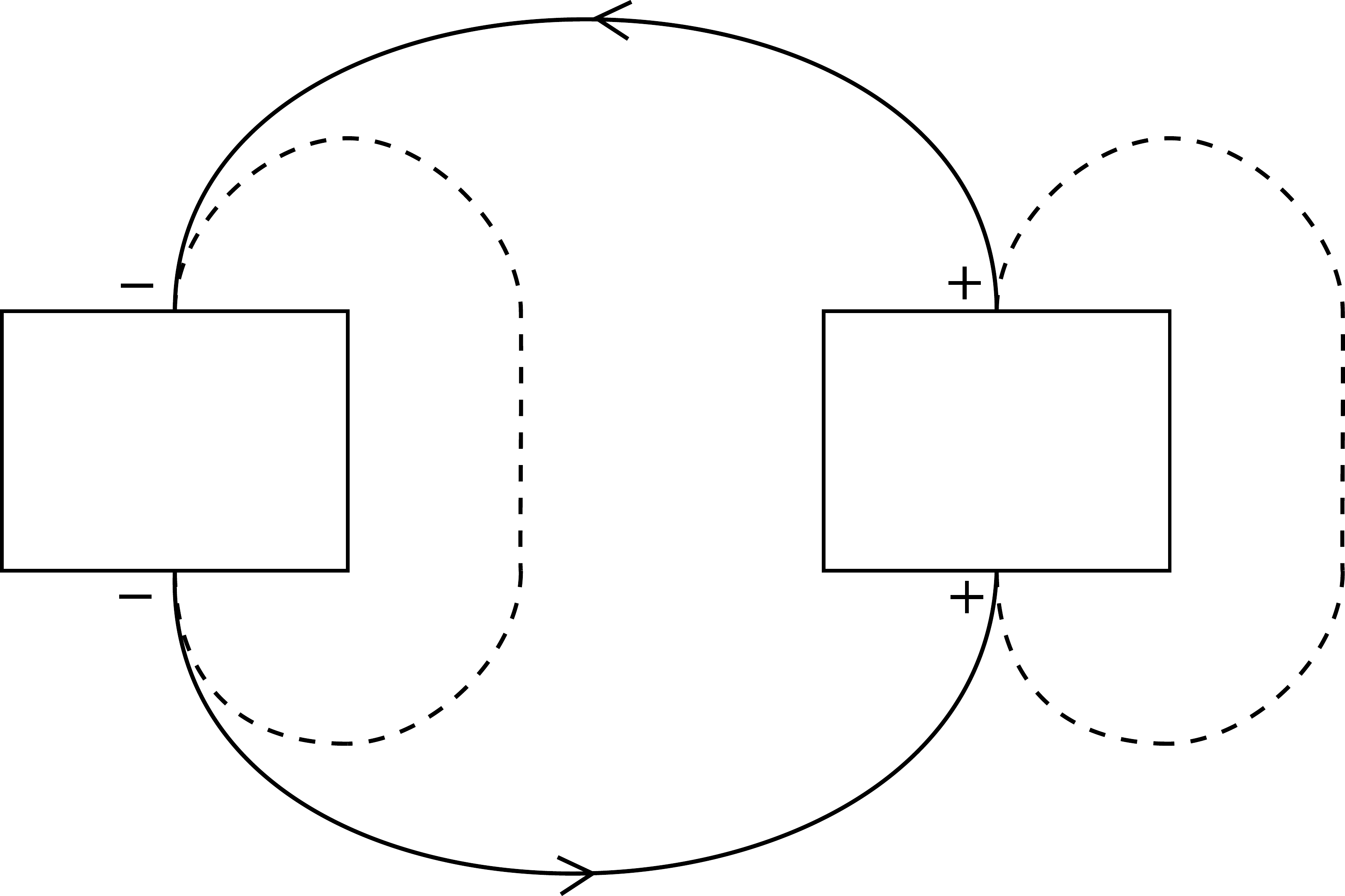}
\caption{Connected sum of knots.}
\label{fig:sum}
\end{center}
\end{figure}
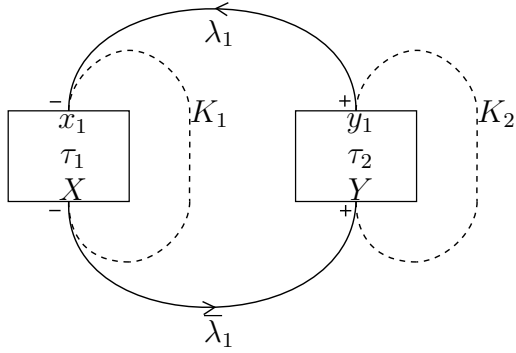

We can assume that $\tau_1$ is oriented from  top to bottom (so both of its boundary points carry the negative sign), while $\tau_2$ is oriented from bottom to top (so both of its boundary points carry the positive sign). Presentation \eqref{eq:ordinary} for the fundamental quandle of the classical closure of a tangle then specializes to the following presentations: 
\begin{align*}Q(K_1)=&\langle x_1,\ldots,x_k,X\mid r_1,\ldots r_h,x_1=X\rangle, \\ \quad Q(K_2)=&\langle y_1,\ldots,y_n,Y\mid s_1,\ldots s_m,y_1=Y\rangle,
\end{align*}
where $x_1=i^+_{\tau_1}(a_1)$, $X=i^-_{\tau_1}(a_1)$, $y_1=i^+_{\tau_2}(b_1)$ and $Y=i^-_{\tau_2}(b_1)$ are the arcs depicted in Figure \ref{fig:sum} and $Q(P_1^-)=F(a_1), Q(P_1^+)=F(b_1)$. 
Using these presentations and applying Corollary \ref{cor2}, we obtain the following presentation for the fundamental quandle of $K_1\sharp K_2$:  
\begin{align}
\label{eq:composite}
  Q(K_1\sharp K_2)=\langle& x_1,\ldots,x_k, y_1\ldots, y_n,X,Y\mid r_1,\ldots,r_h, s_1,\ldots, s_m,  x_1=y_1, X=Y\rangle.
  \end{align}

\begin{figure}[h!]
\labellist
\pinlabel $x_1$ at 40 140
\pinlabel $x_2$ at -5 20
\pinlabel $x_3$ at -10 160
\pinlabel $X$ at 100 20 
\pinlabel $Y$ at 230 20 
\pinlabel $y_1$ at 290 140
\pinlabel $y_3$ at 335 160 
\pinlabel $y_2$ at 335 20 
\normalsize \hair 2pt
\endlabellist
\begin{center}
\includegraphics[scale=0.45]{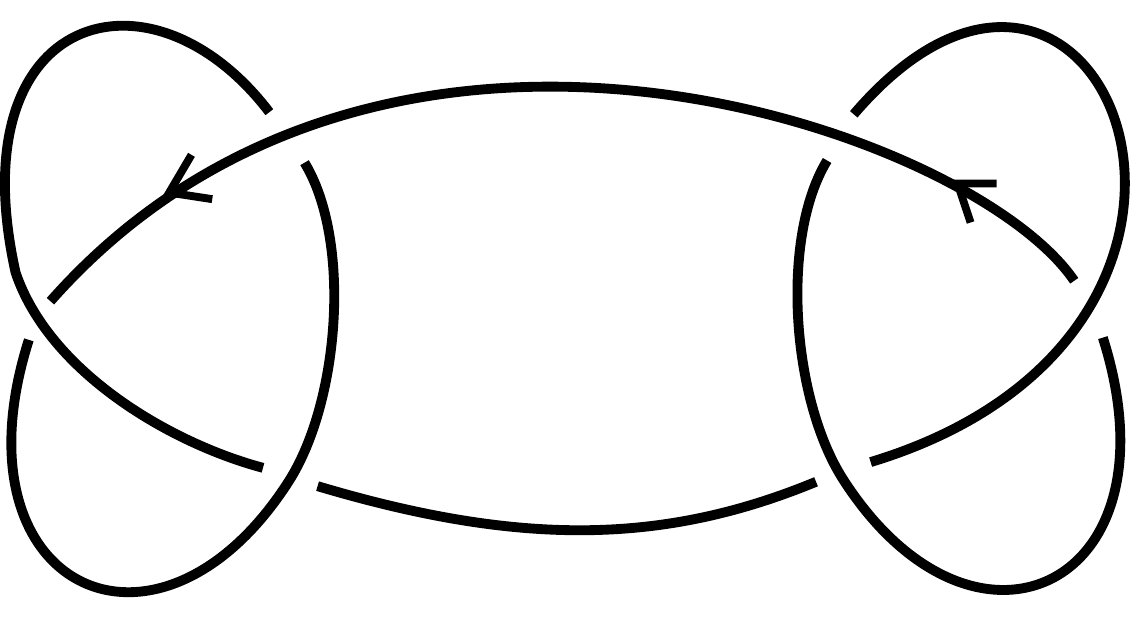}
\caption{The granny knot.}
\label{fig:granny}
\end{center}
\end{figure}
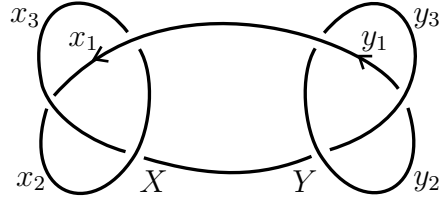

\begin{example}[Fundamental quandle of the granny knot]\label{ex2} 
Consider the knot depicted in Figure \ref{fig:granny}: its name is the granny knot and it is obtained as a connected sum of the left-handed trefoil knot $3_1$  with its mirror image $3_1^*$ (the so called right-handed trefoil knot). Using the left (respectively right) part of the diagram in Figure \ref{fig:granny}, we obtain the following presentations
\begin{xalignat*}{1}
& Q(3_1)=\langle x_1,x_2,x_3,X\, |\, x_3 \tr x_1 =x_2,\, x_2 \tr x_3 = x_1,\, X\tr x_2 = x_3,\, x_1 = X\rangle, \\
& Q(3_1^*)=\langle y_1,y_2,y_3,Y\, |\, y_3 \tr y_1 =y_2,\, Y \tr y_2 = y_3,\, y_2 \tr y_3 = y_1,\, y_1 = Y\rangle.
\end{xalignat*} So the fundamental quandle of the granny knot has a presentation 
\begin{align*}
 \langle & x_1,x_2,x_3,X, y_1,y_2,y_3,Y\mid x_3 \tr x_1 =x_2,\, x_2 \tr x_3 = x_1,\, X\tr x_2 = x_3, y_3 \tr y_1 =y_2,\,\\ & \ Y \tr y_2 = y_3,\, y_2 \tr y_3 = y_1,\, y_1 = x_1,\, X=Y\rangle, 
\end{align*} that reduces to 
\begin{align*}
 \langle & x_2,x_3,y_2,y_3 \mid x_3 \tr (x_2 \tr x_3) =x_2,\, (y_3 \tl y_2)\tr x_2 = x_3, y_3 \tr (x_2 \tr x_3) = y_2,\\ &\  x_2 \tr x_3 = y_2\tr y_3 \rangle.
\end{align*}
\end{example}
 
 Using  presentation \eqref{eq:composite}, a relationship between the counting invariants of $K_{1}$, $K_{2}$ and that of $K_{1}\sharp K_{2}$ may be obtained, as it was done in \cite{CS}. We recall that a \textit{coloring} of a tangle $\tau$ by a finite quandle $T$, or a $T$-\textit{coloring}, is a quandle homomorphism $f\colon Q(\tau)\to T$. Denote by $Col_{T}(\tau)$ the set of all colorings of a tangle $\tau$ by a fixed quandle $T$. The number of all colorings $|Col_{T}(\tau)|$ is called the \textit{counting invariant} of $\tau$ with respect to $T$. If the coloring quandle $T$ has some special properties, the counting invariant of a composite knot $K_{1}\sharp K_{2}$ may be derived from the counting invariants of $K_{1}$ and $K_{2}$. 

\begin{definition} Let $Q$ be a quandle. Denote by $\alpha \colon Q\to Aut(Q)$ the map, given by $\alpha(y)=\alpha_y$ for every $y\in Q$ (see Definition \ref{def:quandle}). The quandle $Q$ is \textit{connected} if $\alpha(Q)$ acts transitively on $Q$. The quandle $Q$ is called \textit{faithful} if the map $\alpha$ is an injection. 
\end{definition}

A $T$-coloring of a tangle $\tau$ can be used to color its endpoints, i.e., the elements of $\partial^+\tau\cup\partial^-\tau$: given a point $(p_i,1)$ in $\partial^+ \tau$, the \textit{color} of $(p_i,1)$ is $f\circ i^+_{\tau}(a_i)$, where $a_i$ is the generator of $Q(\partial^+\tau)$, corresponding to the point $(p_i,1)$ ($a_i$ is the homotopy class of the path $\alpha _{i}$, depicted in Figure \ref{fig_disk}). An analogous definition can be given for points in $\partial^-\tau$ by replacing $+$ with $-$.

\begin{lemma}\cite[Lemma 5.6]{TN} \label{lemmaF} Let $T$ be a finite faithful quandle. Then for any $T$-coloring of a $(1,1)$-tangle, the colors at the two endpoints of the tangle agree. 
\end{lemma}
It follows from this lemma that if $T$ is a faithful quandle, every coloring of a $(1,1)$-tangle $\tau $ induces a coloring of the link $\widehat{\tau }$ and vice-versa. Applying this to tangles $\tau _1 $ and $\tau _2 $, such that $K_i=\widehat{\tau_i}$ for $i=1,2$, we obtain the following restatement of results  \cite[Lemma 3.4, Lemma 3.6]{CS}:

\begin{lemma}Let $T$ be a finite connected faithful quandle, then $$|T|\, |Col_{T}(K_{1}\sharp K_{2})|=|Col _{T}(K_{1})|\, |Col_{T}(K_{2})|\;.$$
\end{lemma} 
\begin{proof}Let $K$ be any knot and fix a diagram $D_K$ of $K$. To each arc of the diagram $D_K$ there corresponds an element $x$ in $Q(K)$:  denote by $Col_{(T,a)}(D_K,x)$ the set of all colorings of the diagram $D_K$ such that the arc labeled by $x$ receives the color $a\in T$, i.e., the elements $f$ of $Col_{T}(K)$ such that  $f(x)=a$. Then we have $Col_{T}(K)=\sqcup _{a\in T}Col_{(T,a)}(D_K,x)$. Choose any $a,b\in T$. Since $T$ is connected, there exists an element $c\in T$ such that $a\tr c=b$. Define a map $\phi \colon Col_{(T,a)}(D_K,x)\to Col_{(T,b)}(D_K,x)$ by $\phi (f)(y)=f(y)\tr c$ and observe that $\phi $ is a bijection. Therefore $|Col_{(T,a)}(D_K,x)|= |Col_{(T,b)}(D_K,x)|$ and consequently $|Col_{T}(K)|=|T| |Col_{(T,a)}(D_K,x)|$. 

For two knots $K_{1}$ and $K_{2}$, let $\tau_i$ be a (1,1)-tangle such that $\widehat{\tau_i}=K_i$, for $i=1,2$. Then the fundamental quandle of $K_1\sharp K_2=\tau_1\tau_2$ admits the presentation \eqref{eq:composite} that stems from presentations $Q(\tau _1 )=\langle x_1,\ldots,x_k,X\mid r_1,\ldots r_h\rangle $ and $Q(\tau _2 )=\langle y_1,\ldots,y_n,Y\mid s_1,\ldots s_m\rangle $, see Figure \ref{fig:sum}. For every coloring $f\in Col_{T}(K_{1}\sharp K_{2})$, the restriction $f|_{\tau _{i}}$ for $i=1,2$ induces a coloring of the tangle $\tau _{i}$, that in turn induces a coloring of $K_{i}$ by Lemma \ref{lemmaF}. Both induced co\-lo\-rings agree at the endpoints of the $(1,1)$-tangles: $f(x_{1})=f(y_1 )=f(X)=f(Y)$. Conversely, for any two colorings $f_{1}\in Col_{T}(K_{1})$ and $f_{2}\in Col_{T}(K_{2})$ such that $f_{1}(x_{1})=f_{2}(y_{1})$, the map $f\colon Q(K_{1}\sharp K_{2})\to T$, given by $f(x)=f_{i}(x)$ for any $x\in Q(K _{i})$ with $i=1,2$, induces a coloring of the composite knot $K_{1}\sharp K_{2}$. Thus, we have a bijection $$\bigcup _{a\in T}\left (Col_{(T,a)}(D_{K_{1}},x_{1})\times Col_{(T,a)}(D_{K_{2}},y_{1})\right )\to Col_{T}(K_{1}\sharp K_{2})\;,$$ where $D_{K_i}$ denotes the diagram of $K_{i}$ for $i=1,2$. It follows that $|T|\cdot \frac{|Col_{T}(K_{1})|}{|T|} \cdot \frac{|Col_{T}(K_{2})|}{|T|}=|Col_{T}(K_{1}\sharp K_{2})|$. 
\end{proof}
 
 \end{subsection}

\begin{subsection}{Satellite knots} 

\label{appl3}

A \textit{satellite knot} is a knot $K$ in $S^{3}$, whose complement contains an incompressible torus that is not parallel to the boundary of the regular neighborhood of $K$. \\

Given an oriented  satellite knot $K$, it is possible to associate to $K$ a couple of oriented knots, the companion and the embellishment, as follows. Choose an incompressible, non boundary parallel torus $T$ in the complement of $K$. Let $V$ be the solid torus, bounded by $T$: the core of $V$ (associated with $T$) is called the \textit{companion} of $K$ and is denoted by $C$. Note that $K\subset V$, and orient $C$ so that $K$ is homologous in $V$ to $wC$ with $w$ a non-negative integer number (there a choice to be made when $w=0$). The number $w$ is called \textit{the winding number} of $K$. Let $U$ be an unknotted solid torus $U\subset S^3$ and let  $f\colon V\to U\subset S^3$ be an orientation and longitude preserving  homeomorphism: the knot $E=f(K)$ is called the \textit{embellishment} of $K$. \\

Various invariants of $K$ may be expressed in terms of $C$ and $E$: for example in \cite{LM}, the Alexander module is computed in such a  way. Analogously, we would like to express the fundamental quandle of a satellite knot $K$  in terms of the fundamental quandles of its companion and  embellishment. In order to do this, we need to introduce the following definition. 

\begin{definition}
 Let $\tau$ be an $(n,m)$-tangle with $\chi$ components. For each point $(p_i,k)$ for $k=0,1$  in $\partial \tau$, consider disjoint closed intervals  $I(p_i,k)\subset ([-1,1]\times\{k\})$ having $(p_i,k)$ as a midpoint and not containing any other point $(p_j,k)$ for $i\neq j$.  For each  component $\tau_i$ of $\tau$ with endpoints $(p_r,k)$, $(p_j,h)$ with $k,h\in\{0,1\}$, consider an embedding $f_i:I\times I\to D\times I$ such that: 
 \begin{itemize}
  \item $f_i(\{\frac12\}\times I)=\tau_i$,
  \item $f_i(I\times\{0\})=I(p_r,k)$ and    $f_i(I\times\{1\})=I(p_j,h)$, 
  \item the projections of  $f_i(\{0\}\times I)$ and $\tau_i$ onto $[-1,1]\times I$ do not intersect,
  \item $f_i(I\times I)\cap f_j(I\times I)=\emptyset$ for $i\ne j$. 
 \end{itemize}
    Given a $(\phi,\psi)$-tangle $\tau$, for each natural number $N\geq 2$ we define the $N$-\textit{cable} of $\tau$, denoted by $N\tau$, as the $(N\phi,N\psi)$-tangle given by $\cup_{i=1}^{\chi} f_i(\{t_1,\ldots, t_N\}\times I)$, where $t_1,\ldots, t_N$ are points in $I$ such that $0=t_1<t_2<\cdots<t_{N-1}<t_N=1$ (see Figure \ref{fig_cable} for an example with $N=2$). 
    
\end{definition}

\begin{figure}[h!]
\labellist
\pinlabel $(p_{1},1)$ at 10 390
\pinlabel $(p_{2},1)$ at 130 390
\pinlabel $(p_3,1)$ at 250 390
\pinlabel $(p_4,1)$ at 360 390
\pinlabel $(p_1,0)$ at 10 -10
\pinlabel $(p_2,0)$ at 130 -10
\pinlabel $I(p_{1},1)$ at 580 390
\pinlabel $I(p_{2},1)$ at 700 390
\pinlabel $I(p_3,1)$ at 820 390
\pinlabel $I(p_4,1)$ at 930 390
\pinlabel $I(p_1,0)$ at 580 -20
\pinlabel $I(p_2,0)$ at 700 -20
\normalsize \hair 2pt
\endlabellist
\begin{center}
\includegraphics[scale=0.35]{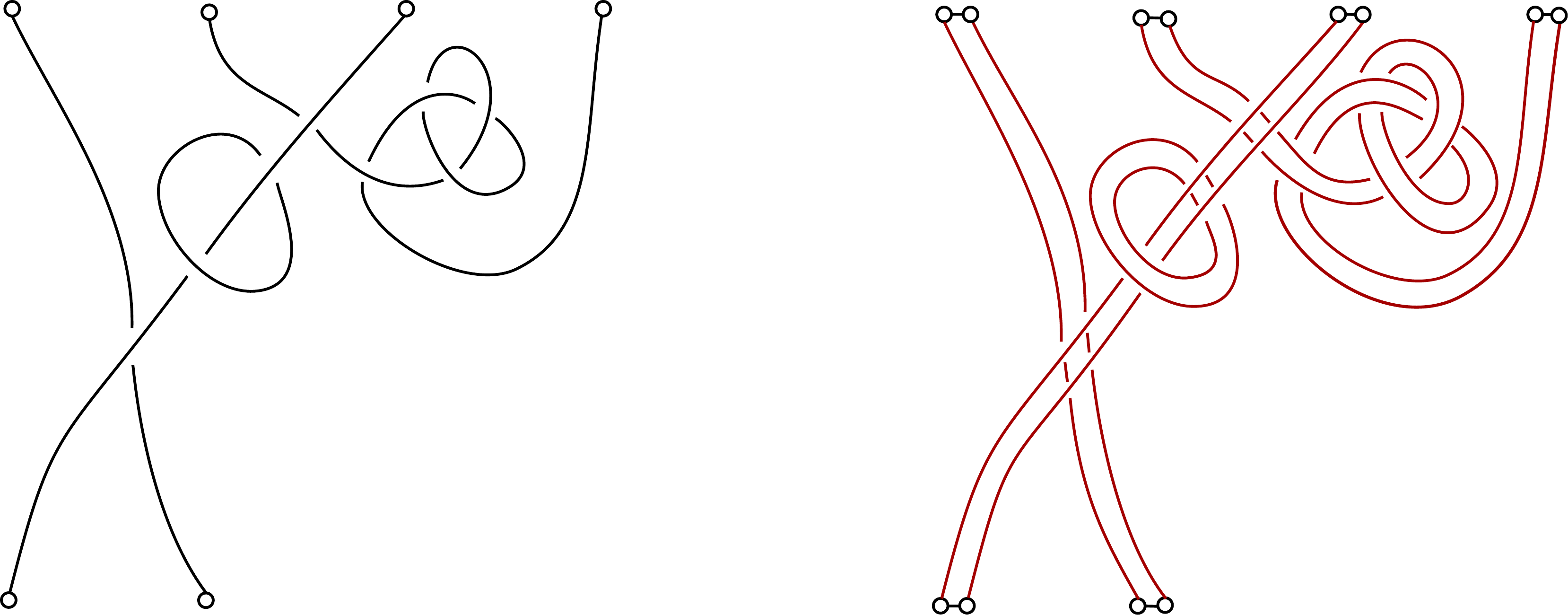}
\vspace{5mm}
\caption{A tangle $\tau $ (left) and its 2-cable $2\tau$ (right).}
\label{fig_cable}
\end{center}
\end{figure}
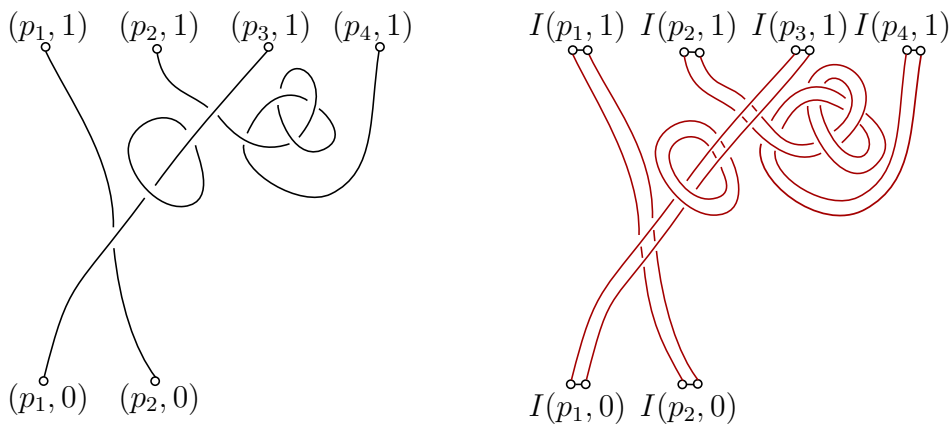

Note that $T_k=kT_1$ and $\lambda_k=k\lambda_1$. \\

Let $K$ be an oriented satellite knot and let $T$ be an incompressible, non boundary parallel torus in the complement of $K$. As above, let $V$ be the solid torus bounded by $T$, let $C$ be the companion of $K$ (associated to $T$) and let $f:V\to U$ with $f(K)=E$.  

We choose a parametrization of the trivial solid torus $U=D\times S^1$ that satisfies the following:
\begin{enumerate}
\item[(i)] $f(C)$ is the oriented curve $t\mapsto (0,e^{2i\pi t})\in  \{0\}\times S^1$ for $t\in [0,1]$,
\item[(ii)] there exists a $t_1\in (0,1)$ so that $D\times [t_1,1]\cap E$ is a trivial tangle $T_k$  in the solid cylinder $D\times [t_1,1]$ for some $k\geq 1$.
\item[(iii)] the two discs $D_0=D\times\{1\}$ and $D_1=D\times\{e^{2i\pi t_1}\}$ intersect $E$ transversely.
\end{enumerate}
It follows that $\tau_E=E\cap D\times[0,t_1]$ is a $(\phi,\phi)$-tangle, whose orientation is induced by that of $E$, and assuming that $D_0$ is the bottom disc of the cylinder, we have $E=\widehat{\tau_E}$. Note that the tangle $\tau_{f(C)}=D\times [0,t_1]\cap f(C)$ is $T_1$ and that the sum of signs of $\phi$ equals the winding number of $K$.

Consider a solid torus, standardly embedded in $S^3$ and containing $V$, such that $f^{-1}(D_0)$ and $f^{-1}(D_1)$ are contained  in meridional disks and  $C=\widehat{\tau_C}$, where $\tau_C$ is the $(1,1)$-tangle $f^{-1}(\tau_{f(C)})$: clearly $k\tau_C=f^{-1}(T_k)$. Inside the solid cylinder, $f^{-1}(\tau_E)$ is isotopic to $\tau_E$. Therefore we obtain $K=\lambda_k^{\eta}(\tau_E\otimes k\tau_C)\overline{\lambda^{\eta}_k}$, where $k\tau_{C}$ is a $(\psi,\psi)$-tangle with $\psi(p_{k-i+1})=-\phi(p_i)$ for $i=1,\ldots,k$, while $\eta$ coincides with $\phi$ on the first $k$ points and with $\psi$ on the last $k$ points. As a consequence,  we have that 
$$Q(K)=Q(\lambda_k^{\eta})\ast_{(Q(P_{2k}^{\eta}),i_k^+,i_{\tau_E}^-\ast i^-_{k\tau_{C}})} (Q(\tau_E)\ast Q(k\tau_C))\ast_{(Q(P_{2k}^{\eta}),i_{\tau_E}^+\ast i_{k\tau_{C}}^+,i_k^+)}Q(\overline{\lambda_k^{\eta}})\;,$$
where $i_k^+$ is the homomorphism described in Example \ref{ex4},  $BQ(\tau_E)$ is $Q(P_k^{\phi})\stackrel{i_{\tau_E}^-}{\rightarrow}Q(\tau _E)\stackrel{i_{\tau_E}^+}{\leftarrow}Q(P_{k}^{\phi})$ and $BQ(k\tau_C)$ is $Q(P_k^{\psi})\stackrel{i_{k\tau_C}^-}{\rightarrow}Q(k\tau_C)\stackrel{i_{k\tau_C}^+}{\leftarrow}Q(P_{k}^{\psi})$.

\begin{remark}
 Note that  for $k=1$, we obtain $K=\lambda_1(\tau_E\otimes \tau_C)\overline{\lambda}_1$ and so $K=E\sharp C$: indeed, composite knots are a particular case of satellite knots. 
\end{remark}

The following proposition describes how to obtain a presentation of $Q(k\tau)$ given a presentation of $Q(\tau)$ and a choice of signs on the boundary points.


\begin{proposition}\label{prop1}
 Let $\tau$ be an oriented  tangle and consider its regular diagram $D_{\tau}$ in $[-1,1]\times I$. Denote by $x_1,\ldots, x_r$ the arcs (or curves) of $D_{\tau}$  and by $C_1,\ldots, C_s$ its  crossings. Fix  $\vec{\epsilon}\in \{-1,1\}^N$, and let  $\vec{\epsilon}\tau$ be the $N$-cable of $\tau$, oriented so that the $j$-th copy of $\tau$  has the same orientation as $\tau$ if $\epsilon_j=+1$ and the opposite one otherwise. Then we have:
 $$Q(\vec{\epsilon}\tau)=\langle \{x_{ij},\ i=1,\ldots, r, \ j=1,\ldots, N\} \mid \{r_{kj},\  k=1,\ldots, s,\  j=1,\ldots, N\}\rangle,$$
 where $r_{kj}: x_{ij}= \left(\cdots\left(\left( x_{hj}\rhd^{\epsilon_1} x_{l1}\right) \vartriangleright^{\epsilon_2} x_{l2}\right)\vartriangleright^{\epsilon_{3}}\cdots \vartriangleright^{\epsilon_{N-1}} x_{l(N-1)}\right) \vartriangleright^{\epsilon_N} x_{lN}$  if the arcs of the crossing $C_k$ are labeled as  depicted in the left part of Figure \ref{fig_ntangle}, and where $\vartriangleright^{1}=\vartriangleright$,  $\vartriangleright^{-1}=\vartriangleleft$.
\end{proposition}

\begin{figure}[h!]
\labellist
\pinlabel $x_{i}$ at 0 320 
\pinlabel $x_{l}$ at 300 320 
\pinlabel $x_{h}$ at 300 620 
\pinlabel $x_{i_{1}}$ at 700 -20 
\pinlabel $x_{i_2}$ at 640 40 
\pinlabel $x_{i_N}$ at 530 160 
\pinlabel $x_{l1}$ at 1450 490 
\pinlabel $x_{l2}$ at 1330 370 
\pinlabel $x_{l(N-1)}$ at 1120 130
\pinlabel $x_{lN}$ at 960 10  
\pinlabel $x_{h_{1}}$ at 1460 730 
\pinlabel $x_{h_2}$ at 1400 790 
\pinlabel $x_{h_N}$ at 1290 910 
\pinlabel $y_{11}$ at 1275 515 
\pinlabel $y_{21}$ at 1170 630 
\pinlabel $y_{N1}$ at 1050 755 
\pinlabel $y_{12}$ at 1180 425
\pinlabel $y_{22}$ at 1080 545
\pinlabel $y_{N2}$ at 955 660 
\pinlabel $y_{1(N-2)}$ at 1070 280 
\pinlabel $y_{2(N-2)}$ at 910 400 
\pinlabel $y_{N(N-2)}$ at 790 525 
\pinlabel $y_{1(N-1)}$ at 940 140 
\pinlabel $y_{2(N-1)}$ at 780 280 
\pinlabel $y_{N(N-1)}$ at 670 410 
\normalsize \hair 2pt
\endlabellist
\begin{center}
\includegraphics[scale=0.25]{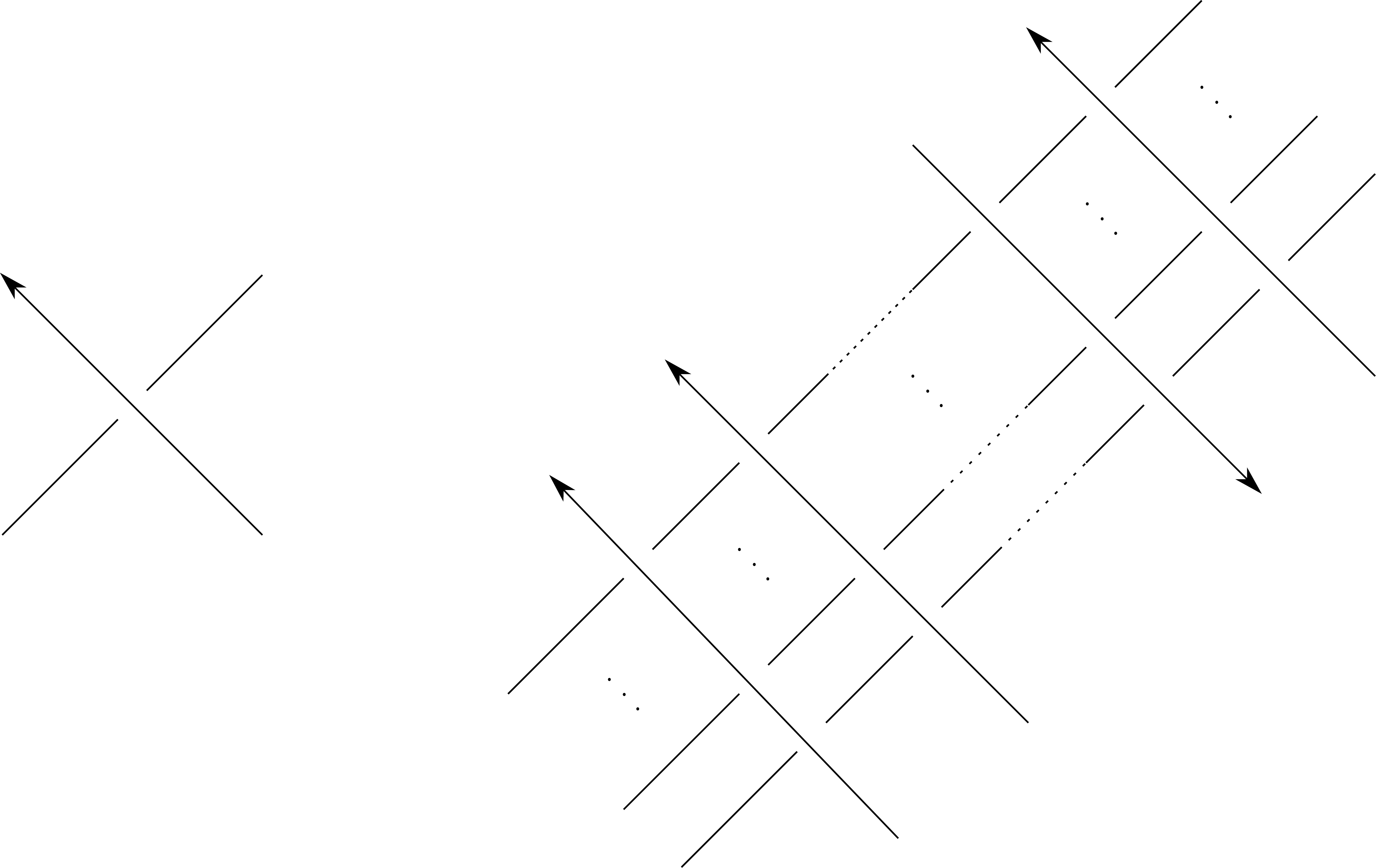}
\vspace{5mm}
\caption{Crossings of $D_{\vec{\epsilon}\tau}$, corresponding to a crossing of $D_{\tau }$.}
\label{fig_ntangle}
\end{center}
\end{figure}
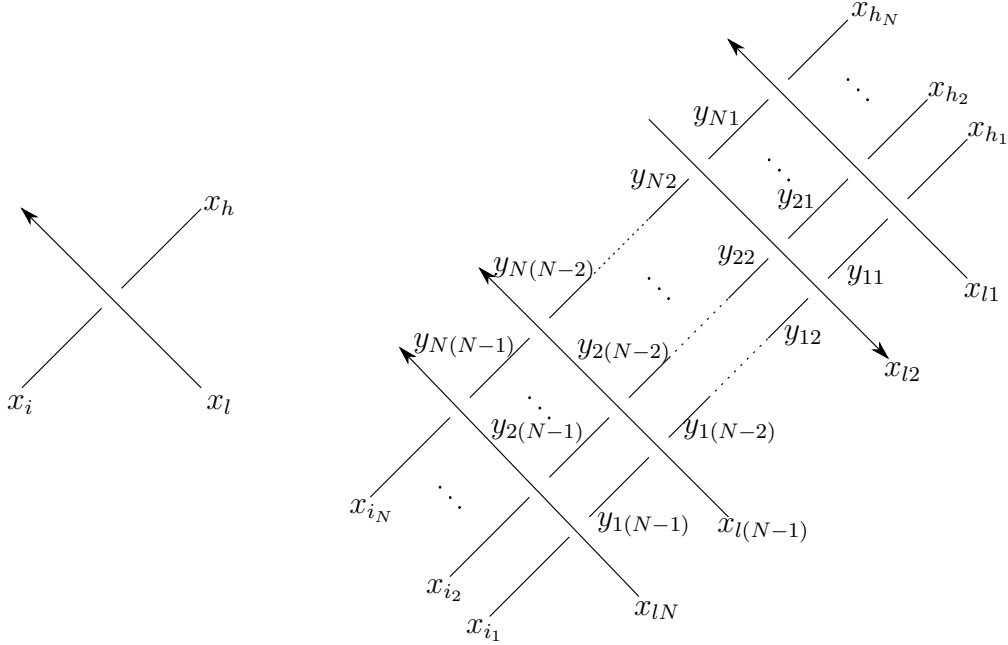

\begin{proof}
A diagram $D_{\vec{\epsilon}\tau}$ in $[-1,1]\times I$ for $\vec{\epsilon}\tau$ is obtained by replacing each arc (or circle) of $D_{\tau}$ with $N$ copies ``parallel'' to each other, where the $j$-th copy has the same orientation of the corresponding arc of $\tau$ if $\epsilon_j=1$ and  the opposite one otherwise (see Figure \ref{fig_ntangle}). To each  crossing of $D_{\tau}$  (on the left of the figure), there correspond $N^2$ crossings in $D_{\vec{\epsilon}\tau}$ (on the right of the figure). To each of the three arcs, say $x_i,x_l,x_h$, involved in a crossing of $D_{\tau}$, there correspond  $N$ arcs in $D_{\vec{\epsilon}\tau}$: we label them  by $x_{ij}$, $x_{lj}$ and $x_{hj}$ for $j=1,\ldots,N$ as denoted in the figure. Moreover,  $D_{\vec{\epsilon}\tau}$ also contains $N^2-N$ new arcs that are labeled by $y_{kb}$ for $k=1,\ldots,N$ and $b=1,\ldots, N-1$. Using the rule \eqref{qrelation}, all those $N^2-N$ new generators could be expressed in terms of the $x_{ij}$'s via relations corresponding to the crossings in which they appear, namely: $y_{j1}=x_{hj}\vartriangleright^{\epsilon_1} x_{l1}$, $y_{jb}=y_{j(b-1)}\vartriangleright^{\epsilon_b} x_{lb}$ for $b=2,\ldots,N-1$ and $x_{ij}=y_{j(N-1)}\vartriangleright^{\epsilon_N} x_{lN}$ for $j=1,\ldots, N$. So the only remaining generators are the $x_{ij}$'s and the relations are $x_{ij}= \left(\cdots\left(\left( x_{hj}\rhd^{\epsilon_1} x_{l1}\right) \vartriangleright^{\epsilon_2} x_{l2}\right)\vartriangleright^{\epsilon_{3}}\cdots \vartriangleright^{\epsilon_{N-1}} x_{l(N-1)}\right) \vartriangleright^{\epsilon_N} x_{lN}$.  As a  result, we obtain the presentation of $ Q(\vec{\epsilon}\tau)$ stated above. 
\end{proof}

Using Proposition \ref{prop1} together with Corollary \ref{cor2}, a presentation for the fundamental quandle of a satellite knot $K=\lambda_k^{\eta}(\tau_E\otimes k\tau_C)\overline{\lambda^{\eta}_k}$ may be computed from the fundamental quandles of $\tau_E$ and $\tau_C$. In this case, the vector $\vec{\epsilon}\in\{-1,1\}^k$ has to be chosen so that $\vec{\epsilon}\tau_C$ is a $(\psi,\psi)$-tangle. We illustrate our results by two explicit calculations.

\begin{figure}[h!]
\labellist
\pinlabel $\tau _E$ at -10 830
\pinlabel $f(\tau _C)$ at 540 830
\pinlabel $(1,1)\tau _C$ at 2040 830
\pinlabel $\lambda _2$ at 900 -45
\pinlabel $\overline{\lambda }_2$ at 900 1650
\pinlabel $y_4$ at 110 1100
\pinlabel $y_5$ at 405 1100
\pinlabel $y_3$ at 405 900
\pinlabel $y_1$ at 90 500
\pinlabel $y_2$ at 405 500
\pinlabel $x_{12}$ at 1780 380
\pinlabel $x_{11}$ at 2100 380
\pinlabel $x_{22}$ at 1150 870
\pinlabel $x_{21}$ at 1415 1000
\pinlabel $x_{32}$ at 1515 540
\pinlabel $x_{31}$ at 1255 540
\pinlabel $x_{42}$ at 1490 1100
\pinlabel $x_{41}$ at 1740 1230
\pinlabel $x_{52}$ at 1290 1300 
\pinlabel $x_{51}$ at 1515 1470
\pinlabel $\otimes $ at 900 830
\normalsize \hair 2pt
\endlabellist
\begin{center}
\includegraphics[scale=0.15]{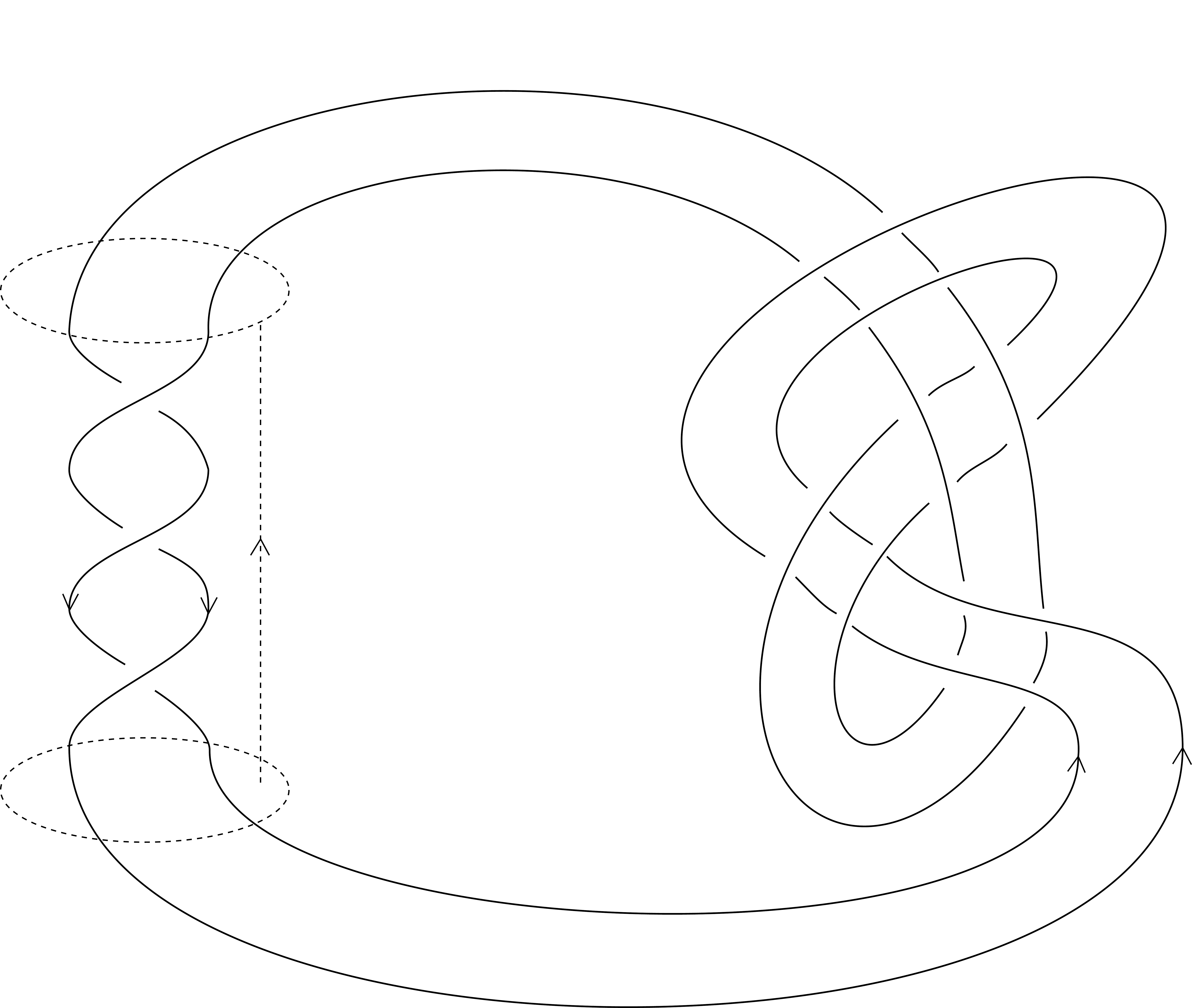}
\vspace{5mm}
\caption{The $(3,2)$-cable of the figure-eight knot and its tangle decomposition.}
\label{fig_cable1}
\end{center}
\end{figure}
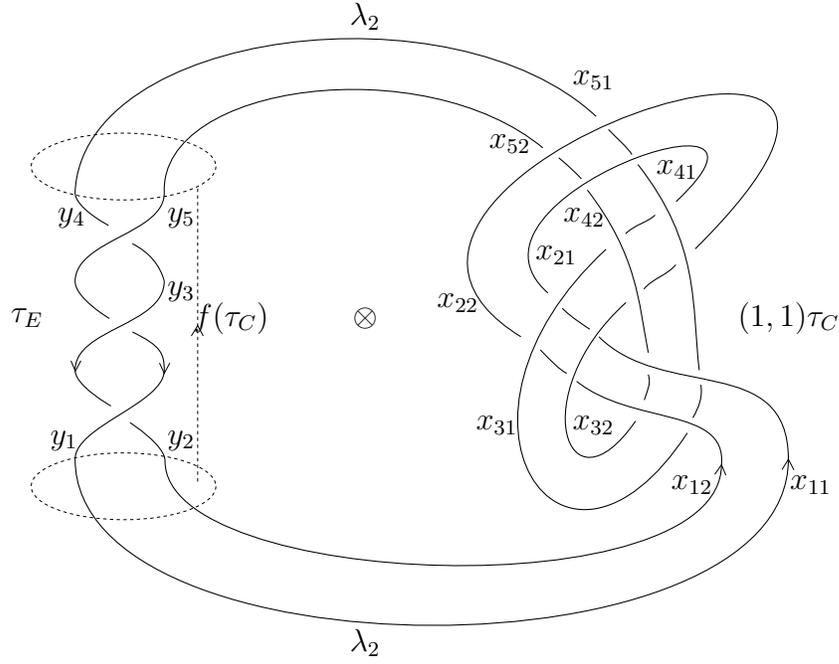

\begin{example}[$(3,2)$-cable of the figure-eight knot] We decompose this satellite knot $K$ into tangles as $K=\lambda _2 (\tau _E \otimes (1,1)\tau _ C)\overline{\lambda}_2$, with orientation indicated in Figure \ref{fig_cable1}. The fundamental quandles of tangles $\tau _E$ and $(1,1)\tau _C$ admit presentations 
\begin{align*}
 Q(\tau _E)&=\langle y_1 ,y_2 ,y_3 ,y_4 ,y_5 \, |\, y_3 \tr y_1 =y_2, y_5 \tr y_3 =y_1 , y_4 \tr y_5 =y_3 \rangle, \\
 Q((1,1)\tau _C)&=\langle \{x_{ij},\  i=1,2,3,4,5,\, j=1,2\}\mid 
 \{(x_{4j}\tr x_{11})\tr x_{12}=x_{3j}, \quad (x_{2j}\tr x_{31})\tr x_{32}=x_{1j}, \\
& \qquad (x_{4j}\tr x_{21})\tr x_{22}=x_{5j}, \quad (x_{2j}\tr x_{41})\tr x_{42}=x_{3j},\  j=1,2\}\rangle. 
\end{align*} 

It follows that the fundamental quandle of the satellite $K$ is given by the presentation
\begin{align*}
 Q(K)=\langle &\{y_i , x_{ij},\  i=1,\ldots ,5 , j=1,2\}\mid
 \{y_1 =x_{11}, y_2 =x_{12}, y_4 = x_{51},  y_5 = x_{52}, \\ &\ y_3 \tr y_1 =y_2, y_5 \tr y_3 =y_1 , y_4 \tr y_5 =y_3 , (x_{4j}\tr x_{11})\tr x_{12}=x_{3j},\\ &\  (x_{2j}\tr x_{31})\tr x_{32}=x_{1j}, (x_{4j}\tr x_{21})\tr x_{22}=x_{5j}, \\ &\ (x_{2j}\tr x_{41})\tr x_{42}=x_{3j},\  j=1,2\}\rangle, 
\end{align*} that reduces to 
\begin{align*}
 Q(K)=\langle& y_1 ,y_2 ,x_{21},x_{22},x_{41},x_{42}\, |\, (x_{21}\tr ((x_{41}\tr y_{1})\tr y_{2}))\tr ((x_{42}\tr y_{1})\tr y_{2})=y_{1},\\
&  (x_{22}\tr ((x_{41}\tr y_{1})\tr y_{2}))\tr ((x_{42}\tr y_{1})\tr y_{2})=y_{2},\\ & (x_{41}\tr x_{21})\tr x_{22}=(y_2 \tl y_1 )\tl (y_{1}\tl (y_{2}\tl y_{1})),\\
& (x_{42}\tr x_{21})\tr x_{22}=y_{1}\tl (y_{2}\tl y_{1}), (x_{41}\tr y_{1})\tr y_{2}=(x_{21}\tr x_{41})\tr x_{42}, \\
& (x_{42}\tr y_{1})\tr y_{2}=(x_{22}\tr x_{41})\tr x_{42}\rangle.
\end{align*}

\end{example}

\begin{figure}[h!]
\labellist
\pinlabel $\tau _E$ at -10 900
\pinlabel $f(\tau _C)$ at 570 900
\pinlabel $(1,1)\tau _C$ at 2635 900
\pinlabel $\lambda _2$ at 1000 -55
\pinlabel $\overline{\lambda }_2$ at 1000 1820
\pinlabel $y_3$ at 80 1160
\pinlabel $y_4$ at 410 1160
\pinlabel $y_1$ at 80 600
\pinlabel $y_2$ at 400 600
\pinlabel $x_{12}$ at 1710 390
\pinlabel $x_{11}$ at 2090 350
\pinlabel $x_{22}$ at 2290 1310
\pinlabel $x_{21}$ at 2610 1300
\pinlabel $x_{32}$ at 1700 840
\pinlabel $x_{31}$ at 1360 840
\pinlabel $x_{42}$ at 1810 950
\pinlabel $x_{41}$ at 2010 1130
\pinlabel $\otimes $ at 1000 900
\normalsize \hair 2pt
\endlabellist
\begin{center}
\includegraphics[scale=0.12]{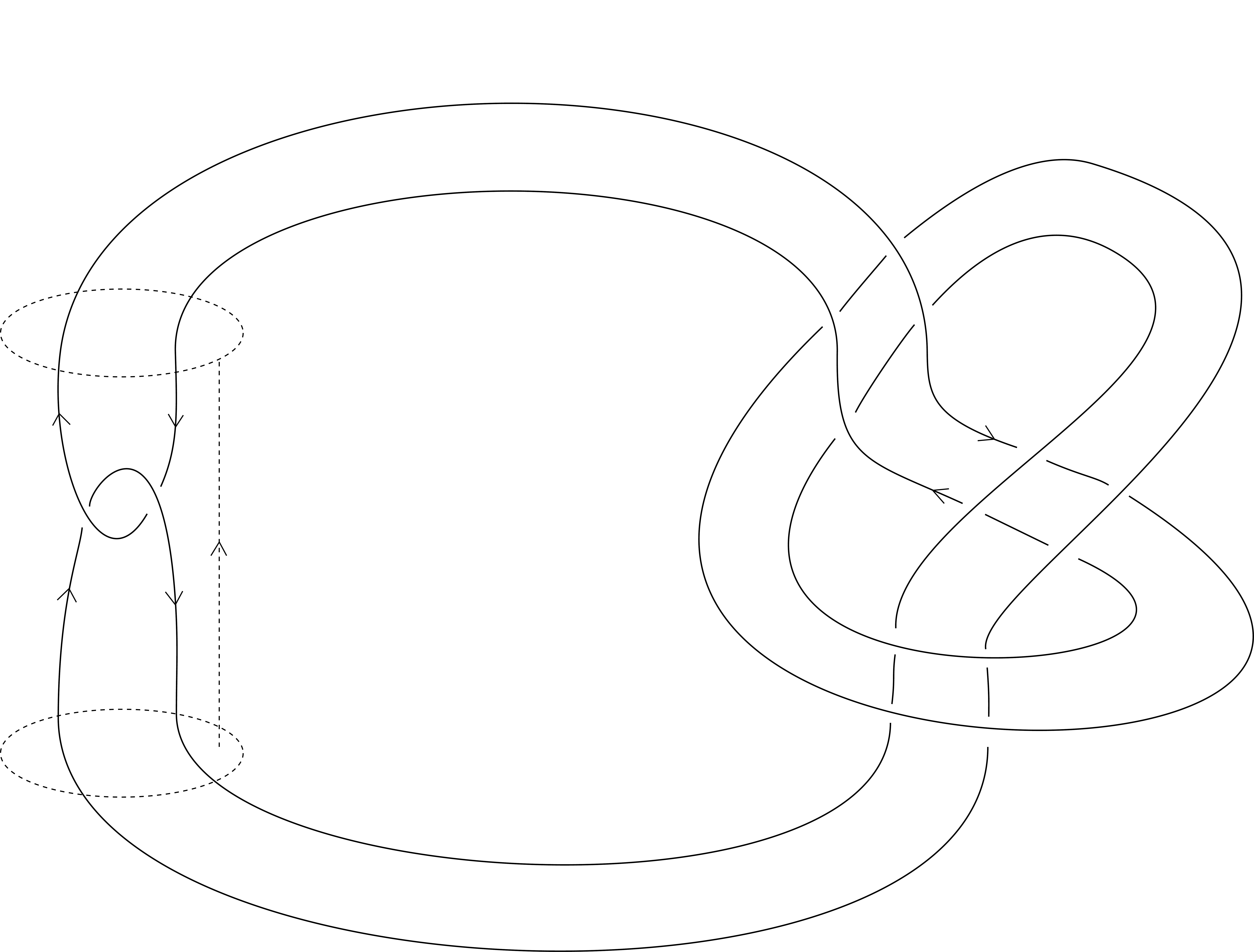}
\vspace{5mm}
\caption{The double of the trefoil knot and its tangle decomposition.}
\label{fig_double}
\end{center}
\end{figure}

\begin{example}[Double of the trefoil knot] The double $D$ of the trefoil knot may be decomposed into tangles as $D=\lambda _2 (\tau _E \otimes (1,-1)\tau _ C)\overline{\lambda }_2$ with orientation indicated in Figure \ref{fig_double}. The fundamental quandles of tangles $\tau _E$ and $(1,-1)\tau _C$ admit presentations 
\begin{align*}
Q(\tau _E)&=\langle y_1 ,y_2 ,y_3 ,y_4 \, |\, y_2 \tr y_3 =y_1, y_3 \tr y_2 =y_4 \rangle, \\
Q((1,-1)\tau _C)&=\langle \{x_{ij},\  i=1,2,3,4,\, j=1,2\}\, |\, \\
& \quad \{(x_{1j}\tl x_{31})\tr x_{32}=x_{2j}, (x_{3j}\tl x_{21})\tr x_{22}=x_{4j}, (x_{2j}\tl x_{41})\tr x_{42}=x_{3j},\ j=1,2\}\rangle. 
\end{align*}

It follows that the fundamental quandle of the satellite $D$ is given by the presentation
\begin{align*}
 Q(D)=\langle& \{y_i , x_{ij}\, |\, i=1,2,3,4 , j=1,2\}\, |\{y_1 =x_{11}, y_2 =x_{12}, y_3 = x_{41}, y_4 = x_{42},\\ & y_2 \tr y_3 =y_1, y_3 \tr y_2 =y_4  (x_{1j}\tl x_{31})\tr x_{32}=x_{2j}, (x_{3j}\tl x_{21})\tr x_{22}=x_{4j},\\ & (x_{2j}\tl x_{41})\tr x_{42}=x_{3j}, j=1,2\}\rangle, 
\end{align*} that reduces to 
\begin{align*}
 Q(D)=\langle &y_2 ,y_3 ,x_{31},x_{32}\mid
 (x_{31}\tl (((y_{2}\tr y_{3})\tl x_{31})\tr x_{32}))\tr ((y_{2}\tl x_{31})\tr x_{32})=y_{3},\\ &  ((((y_{2}\tr y_{3})\tl x_{31})\tr x_{32})\tl y_{3})\tr (y_{3}\tr y_{2})=x_{31},\\
& (x_{32}\tl (((y_{2}\tr y_{3})\tl x_{31})\tr x_{32}))\tr ((y_{2}\tl x_{31})\tr x_{32})=y_{3}\tr y_2 , \\ & (((y_{2}\tl x_{31})\tr x_{32})\tl y_{3})\tr (y_{3}\tr y_{2})=x_{32}\rangle.
\end{align*}
\end{example}

\end{subsection}

\begin{subsection}{Conclusion}
\label{concl}
In this paper, knot quandle decompositions were introduced. We demonstrated how by such decompositions, fundamental quandles of links may be analyzed and their presentations simplified, thereby making the quandle invariants of links more useful. To conclude, we offer some ideas for further research. 
\begin{enumerate}
\setlength{\itemsep}{2pt}
\item By analyzing the fundamental quandle of links using quandle decompositions, is it possible to determine which quandles are the fundamental quandles of links?
\item In Subsection \ref{appl2}, we derived a relationship between the counting invariant of a composite knot and the counting invariants of its factor knots. What can be said about counting invariants (or other quandle invariants) of satellite knots? What about periodic links - is there a relationship between the counting invariant of a periodic link $\widehat{\tau ^{p}}$ and the counting invariant of the tangle $\tau $?   
\item As we mentioned in Remark \ref{rem_fr}, the functor $BQ$ may also be used in the study of braids. To begin with, one could try to understand the representation of the braid group $B_k$ in the automorphism group of the free quandle $Aut(F(P_k))$. Using this representation, new quandle invariants of braids may be defined.
\item The fundamental quandle is a classifying invariant for knots in $S^3$, so in principle every knot invariant in $S^3$ is somehow encoded in the fundamental quandle of the knot (see \cite{FR}). For some invariants, such as the Alexander polynomial, the relationship is well known (see \cite{FR}), while for others it is not. The fundamental quandle functor could be useful for studying those invariants that may be defined using a tangle approach, such as the Jones polynomial or the quantum invariants of links. 
\item Several generalizations of the quandle structure as virtual quandles (see \cite{Man}) or biquandles (see \cite{H,KAUF}), have been introduced to study classical knots and links, virtual ones or knots and links in manifolds different from $S^3$  (see \cite{CN}). Is it possible to extend the fundamental quandle functor in these cases?
\end{enumerate}
\end{subsection}

\end{section}


\begin{thebibliography}{1}


\bibitem{AM} K. Ameur, M. Elhamdadi, T. Rose, M. Saito \and C. Smudde, \textit{Tangle embeddings and quandle cocycle invariants} Experiment. Math.  \textbf{17} (2007), 487--497. 

\bibitem{BL} 
C. Blanchet, N. Geer, B. Patureau-Mirand and N. Reshetikhin, \textit{Holonomy braidings, biquandles and quantum invariants of links with $SL_{2}(\mathbb{C})$ flat connections} preprint (2018), arXiv:1806.02787.

\bibitem{BZ} G. Burde and H. Zieschang, \textit{Knots}, Walter de Gruyter,  2003. 

\bibitem{CM} A. Cattabriga and E. Manfredi, \textit{Diffeomorphic vs Isotopic Links in Lens Spaces} Mediterr. J. Math.  \textbf{15}:172 (2018). 

\bibitem{CN}  A. Cattabriga and T. Nasybullov, \textit{Virtual quandles for links in lens spaces} RACSAM Rev. R. Acad. Cienc. Exactas Fis. Nat. Ser. A. Mat.  \textbf{112}  (2018), 657--669. 

\bibitem{CS}  W. E. Clark, M. Saito and L. Vendramin, \textit{Quandle coloring and cocycle invariants of composite knots and abelian extensions} J.~Knot Theory Ramifications \textbf{25}  (2016), 1650024.


\bibitem{FR}  R. Fenn and C. Rourke, \textit{Racks and links in codimension two} J.~Knot Theory Ramifications \textbf{1} (1992), 343--406.


\bibitem{DJ}  D. Joyce, \textit{A classifying invariant of knots, the knot quandle} J.  Pure Appl. Algebra \textbf{23} (1982), 37--65.

\bibitem{H} E. Horvat, \textit{The topological biquandle of a link}, to appear in Pacific Journal of Mathematics, 2019.

\bibitem{KAUF} R. Fenn, M. Jordan-Santana and L. Kauffman, \textit{Biquandles and virtual links} Topology  Appl. \textbf{145}  (2004), 157--175.

\bibitem{LM}  C. Livingston and P. Melvin, \textit{Abelian invariants of satellite knots} in Geometry and topology (College Park, Md., 1983/84),  Lecture Notes in Math. \textbf{1167}   (Springer, 1985), 217--227. 

\bibitem{MAT} S. V. Matveev, \textit{Distributive grupoids in knot theory} Math. USSR Sbornik \textbf{47}   (1984), 73--83.

\bibitem{Man} V.~Manturov, \textit{On invariants of virtual links}, Acta Appl. Math. \textbf{72}  (2002), 295--309.



\bibitem{NI1}  M. Niebrzydowski, \textit{On colored quandle longitudes and its applications to tangle embeddings and virtual knots} J.~Knot Theory Ramifications \textbf{15}  (2006), 1049-1059. 

\bibitem{TN}  T. Nosaka, \textit{On homotopy groups of quandle spaces and the quandle homotopy invariant of links} Topology  Appl.  \textbf{158}  (2011), 996-1011.



\end{thebibliography}
\end{document}